\documentclass[reqno,12pt]{amsart}

\usepackage{amsmath, epsfig, cite}
\usepackage{amssymb}
\usepackage{comment}
\usepackage{amsfonts}
\usepackage{latexsym}
\usepackage{amsthm}
\usepackage[colorlinks, linkcolor=blue]{hyperref}
\usepackage{tabularx}
\usepackage{booktabs}
\usepackage{amsmath, epsfig, cite}
\usepackage{amssymb}
\usepackage{amsfonts}
\usepackage{latexsym}
\usepackage{extarrows}
\newtheorem{theorem}{Theorem}[section]

\newtheorem{corollary}[theorem]{Corollary}

\theoremstyle{remark}
\newtheorem{remark}{Remark}

\numberwithin{equation}{section}

\allowdisplaybreaks
\author{Chang Xu$^*$}
\author{Dunkun Yang}
\address{School of Mathematical Sciences, East China Normal University,
Shanghai 200241, P. R. China}
\email{52275500007@stu.ecnu.edu.cn, Yangdunkun@foxmail.com}
\thanks{$^*$ Corresponding author}

\title{Some $q$-transformation formulas and Rogers-Ramanujan type identities}


\begin{document}
\maketitle
\begin{abstract}
In this paper, we explore the role that Liu's transformation formula can play in discovering Rogers-Ramanujan type identities.  
Specifically, we combine Liu's transformation formula with other $q$-series summations to derive a series of parameterized identities.
Thereafter, through careful selection of these  parameters,  we obtain  $75$ Rogers-Ramanujan type identities from Slater's list, and uncover several new Rogers-Ramanujan type identities.
%
%
\end{abstract}
\section{Introduction}
Throughout the paper, we assume that $|q|<1$ for convergence and define the $q$-shifted factorials as
\begin{align*}
\left(a;q\right)_{0}=1,
\quad\left(a;q\right)_{n}=\prod_{k=0}^{n-1}\left(1-aq^k\right)
,\quad
n=1,2,\cdots, \mathrm{or}\, {\infty}.\end{align*}
Also,  following the definition of Slater \cite{Slater}, we define
$\left(q;q\right)_{-1}=1$.
For brevity, we abbreviate a collection of several  $q$-shifted factorials on the same base $q$ as
\begin{equation*}
\left(a_1,a_2,\cdots,a_m;q\right)_n
=\left(a_1;q\right)_{n}\left(a_2;q\right)_{n}\cdots\left(a_m;q\right)_{n}.
\end{equation*} 

Following Gasper and Rahman's definition \cite{Gasper_and_Rahman}, the $_r\phi_s(\cdot)$ basic hypergeometric series is defined as
\begin{equation*}
{_r\phi_s}\left(
\begin{array}{cccc}
a_1,\,a_2,\,\cdots,a_r\\
b_1,\,b_2,\,\cdots,b_s
\end{array}
;q,\,z
\right)
=
\sum_{n=0}^{\infty}
\frac{\left(a_1,\,a_2,\,\cdots,a_r;q\right)_n}
{\left(q,\,b_1,\,b_2,\,\cdots,b_s;q\right)_n}
\left(\left(-1\right)^nq^{n(n-1)/2}\right)^{1+s-r}z^n
\end{equation*}
where $q\neq0$ when $r>s+1$.

Many fantastic sum-product $q$-sereis were uncovered by Rogers \cite{Rogers}, including
\begin{align}
&\sum_{n=0}^{\infty}\frac{q^{n^2}}{\left(q;q\right)_n}
=\frac{1}{\left(q,q^4;q^5\right)_{\infty}},\label{q1}\\
&
\sum_{n=0}^{\infty}\frac{q^{n(n+1)}}{\left(q;q\right)_n}
=\frac{1}{\left(q^2,q^3;q^5\right)_{\infty}}.\label{q2}
\end{align}
However, Rogers's work didn't attract people's attention for some time until Ramanujan independently rediscovered  the identities \eqref{q1}-\eqref{q2}. Thereafter, the identities \eqref{q1}-\eqref{q2} were usually named as the Rogers-Ramanujan identities. Later, MacMahon \cite{MacMahon} gave combinatorial interpretations.  Schur \cite{Schur} independently rediscovered \eqref{q1} and \eqref{q2} along with their combinatorial interpretation.

Rogers-Ramanujan identities play crucial roles in various branches of mathematics and physics. For example, these identities have a close relation to combinatorics, vertex algebras, knot theory, and statistical mechanics.
Therefore, discovering identities similar to \eqref{q1}-\eqref{q2}, commonly  referred to as Rogers-Ramanujan type identities, has attracted considerable attention from experts and led to significant progress. One of the celebrated  works concerning this topic is Slater's list \cite{Slater} which includes $130$ of such identities, such as
\cite[Eqs. (58), (59)]{Slater}
\begin{align*}
\sum_{n=0}^{\infty}\frac{\left(q^2;q^2\right)_{n-1}}
{\left(q;q\right)_{n-1}\left(q;q\right)_n\left(q;q^2\right)_n}q^{n^2}
&{=}\frac{\left(-q^5,-q^7,q^{12};q^{12}\right)_{\infty}}{\left(q;q\right)_{\infty}},
\quad{\tag{S.58}}
\\
\sum_{n=0}^{\infty}\frac{q^{n^2+2n}}
{\left(q;q\right)_n\left(q;q^2\right)_{n+1}}
&{=}\frac{\left(q^2,q^{12},q^{14};q^{14}\right)_{\infty}}
{\left(q;q\right)_{\infty}}.\quad {\tag{S.59}}
\end{align*}
Here and forth, we use the label $\left(\mathrm{S}.n\right)$ to denote the equation $(n)$ in Slater's list \cite{Slater}.

Slater used the Bailey lemma to  discovery these identities. Also, more Rogers-Ramanujan type identities  were uncovered by applying the Bailey lemma. The reader may be referred to \cite{Andrews2,Bowman,McLaughlin-2008, McLaughlin_2008}.   In this paper,  we shall show how to obtain Rogers-Ramanujan type identities  without using the Bailey lemma. 
The key tool in our proof is the following elegant transformation formula for $q$-series introduced by Liu \cite{Liu2}.
 \begin{theorem}\label{theorem1}
 For a complex sequence $\left\{A_n\right\}$, under suitable convergence conditions, there holds
 \begin{align}
 &\frac{\left(\alpha q,\alpha ab/q;q\right)_{\infty}}
 {\left(\alpha a,\alpha b;q\right)_{\infty}}
 \sum_{n=0}^{\infty}A_n\left(q/a;q\right)_n\left(\alpha a\right)^n\nonumber\\
 &\quad=\sum_{n=0}^{\infty}
 \frac{\left(1-\alpha q^{2n}\right)\left(\alpha,q/a,q/b;q\right)_n}{\left(1-\alpha\right)\left(q,\alpha a,\alpha b;q\right)_n}\left(-\alpha ab\right)^nq^{n(n-3)/2}\nonumber\\
 &\qquad\times\sum_{k=0}^{n}\frac{\left(q^{-n},\alpha q^n;q\right)_k}{\left(q/b;q\right)_k}\left(q^2/b\right)^kA_k.
 \label{Formula3}
  \end{align}
 \end{theorem}
Specifically,
we take
\begin{equation*}
A_n=\frac{\left(q/b,\beta,\gamma,\lambda;q\right)_n}
{\left(q, c, d,e,h;q\right)_n}\left(b/q\right)^n
\end{equation*}
into Theorem \ref{theorem1} to obtain the following  transformation formula.
\begin{theorem}\label{thm1}
For $\max\left\{|c|,|d|,|e|,|h|,|\alpha a|,|\alpha b|,|\alpha ab/q| \right\}<1$, we have
\begin{align}
&\frac{\left(\alpha q,\alpha ab/q;q\right)_{\infty}}{\left(\alpha a,\alpha b;q\right)_{\infty}}
{_5\phi_4}\left(\begin{array}{ccccc}
q/a,\,q/b,\,\beta,\,\gamma,\,\lambda\\
c,\,d,\,e,\,h
\end{array}
;q,{\alpha ab}/{q}\right)\nonumber\\
&\quad=
\sum_{n=0}^{\infty}
\frac{\left(1-\alpha q^{2n}\right)\left(\alpha,q/a,q/b;q\right)_n}
{\left(1-\alpha\right)\left(q,\alpha a,\alpha b;q\right)_n}\left(-\alpha ab\right)^nq^{n\left(n-3\right)/2}\nonumber\\
&\qquad\times{_5\phi_4}\left(\begin{array}{ccccc}
q^{-n},\,\alpha q^n,\,\beta,\,\gamma,\,\lambda\\
c,\,d,\,e,\,h
\end{array}
;q,q\right).\label{Formula4}
\end{align}
\end{theorem}
By combining  \eqref{Formula4} with other well-known $q$-series summation formulas and  the Jacobi triple product identity  \cite[Eq. (1.6.1)]{Gasper_and_Rahman}
\begin{equation}\label{Jacobi}
\sum_{n=-\infty}^{\infty}q^{\frac{n^2-n}{2}}z^n
=\left(-z,-q/z,q;q\right)_{\infty},
\end{equation}
we derive a number of  parameter-dependent identities in Sections \ref{sec2}-\ref{sec7}.  In particular, by making specific choices for the parameters, we obtain numerous  Rogers-Ramanujan type identities, including $75$ identities listed by Slater and  the following ones that appear to be new,
\begin{align*}
\sum_{n=0}^{\infty}\frac{1}{\left(q;q\right)_{2n+1}}q^{n^2}&=
\frac{\left(q^8,q^{12},q^{20};q^{20}\right)_{\infty}+q\left(q^4,q^{16},q^{20};q^{20}\right)_{\infty}}
{\left(q;q\right)_{\infty}\left(-q^2;q^2\right)_{\infty}},
\\
\sum_{n=0}^{\infty}\frac{\left(q;q\right)_{3n}\left(-q^3;q^3\right)_n}
{\left(q^3;q^3\right)_n\left(q^3;q^3\right)_{2n}}q^{\frac{3n^2-3n}{2}}
&\!=\!\frac{\left(-q^3;q^3\right)_{\infty}}{\left(q^3;q^3\right)_{\infty}}
\left(\left(q,q^5,q^6;q^6\right)_{\infty}+\left(q^2,q^4,q^6;q^6\right)_{\infty}\right),\\
\sum_{n=0}^{\infty}\frac{1}{\left(-q;q^2\right)_{n}\left(q;q\right)_{2n+1}}q^{2n^2+2n}
&\!=\!\frac{\left(q^8,q^{20},q^{28};q^{28}\right)_{\infty}+
q
\left(q^4,q^{24},q^{28};q^{28}\right)_{\infty}
}{\left(q^2;q^2\right)_{\infty}}.
\end{align*}

Our method is constructive, which means that we are able to  derive the product side of a sum-product $q$-series identity from its given sum side, without requiring prior knowledge of product side.  The same methodology has been applied  to the  study of identities on false theta functions, Hecke-type series and Appell-Lerch series, as documented in \cite{Wang_Chen,WangChun,Wang,Wang_Yee}.



\section{Identities from three ${_2\phi_1}$ summations}\label{sec2}
When  $\lambda=\gamma=h=e$, Theorem \ref{thm1} reduces to  the following result which was first proposed by Liu in \cite{Liu_1} and later applied in \cite{Wang, Wang_Chen,Wang_Yee}: for $|\alpha ab/q|<1$,
\begin{align}
&\frac{\left(\alpha q,\alpha ab/q;q\right)_{\infty}}{\left(\alpha a,\alpha b;q\right)_{\infty}}
{_3\phi_2}\left(\begin{array}{ccc}
q/a,\,q/b,\,\beta\\
c,\,d
\end{array}
;q,{\alpha ab}/{q}\right)\nonumber\\
&\quad=
\sum_{n=0}^{\infty}
\frac{\left(1-\alpha q^{2n}\right)\left(\alpha,q/a,q/b;q\right)_n}
{\left(1-\alpha\right)\left(q,\alpha a,\alpha b;q\right)_n}
\left(-\alpha ab\right)^nq^{{n\left(n-3\right)}/{2}}
{_3\phi_2}\left(\begin{array}{ccc}
q^{-n},\,\alpha q^n,\,\beta\\
c,\,d
\end{array}
;q,q\right).\label{Formula1}
\end{align}
Now, we use the $q$-Chu-Vandermonde summation formula and  the transformation \eqref{Formula1} to derive the following transformation formula.
\begin{theorem}\label{thm9}
For $\max\left\{|\alpha a|,|\alpha b|,|c|,|\alpha ab/q^2|\right\}<1$, we have
\begin{align*}
&\frac{\left(\alpha q^2,\alpha ab/q^2;q^2\right)_{\infty}}
{\left(\alpha a,\alpha b;q^2\right)_{\infty}}
\sum_{n=0}^{\infty}\frac{\left(q^2/a,q^2/b;q^2\right)_n}
{\left(q^2,c;q^2\right)_n}\left(\alpha ab/q^2\right)^n\\
&\quad=\sum_{n=0}^{\infty}\frac{1-\alpha q^{4n}}{1-\alpha}
\frac{\left(\alpha,q^2/a,q^2/b,\alpha q^2/c;q^2\right)_n}
{\left(q^2,\alpha a,\alpha b,c;q^2\right)_n}\alpha^n a^nb^nc^nq^{2n^2-4n}.
\end{align*}
\end{theorem}
\begin{proof}
We begin by recalling  the $q$-Chu-Vandermonde summation formula \cite[Eq. (1.5.3)]{Gasper_and_Rahman}
\begin{equation}\label{q-Chu}
_2\phi_1\left(
\begin{array}{cc}
q^{-n},\,a\\
c
\end{array};q,\,q\right)=\frac{\left(c/a;q\right)_{n}}{\left(c;q\right)_n}a^n.
\end{equation}
Replacing $q$ by $q^2$ and setting $\beta=d$ in \eqref{Formula1}, and invoking \eqref{q-Chu} with $a=\alpha q^{2n}$, we arrive at Theorem \ref{thm9}.
\end{proof}

\begin{corollary}
We have
\begin{align}
\sum_{n=0}^{\infty}\frac{\left(-1\right)^nq^{n^2}}{\left(q^2;q^2\right)_n}
&{=}\left(q;q^2\right)_{\infty},\tag{S.3 and S.23}\label{S.3}\\
\sum_{n=0}^{\infty}\frac{\left(-1\right)^n\left(-q;q^2\right)_n}
{\left(q^4;q^4\right)_n}q^{n^2}
&{=}\left(q;q^2\right)_{\infty}
\left(q^2;q^4\right)_{\infty}.\tag{S.4}\label{S.4}
\end{align}
\end{corollary}
\begin{proof}
When $\alpha=-1, c=-q^2$, Theorem \ref{thm9} reduces to
\begin{align}
\sum_{n=0}^{\infty}\frac{\left(q^2/a,q^2/b;q^2\right)_{n}}{\left(q^2,-q^2;q^2\right)_{n}}
\left(-ab/q^2\right)^n=\frac{\left(-a,-b;q^2\right)_{\infty}}{\left(-q^2,-ab/q^2;q^2\right)_{\infty}}, \label{A14}
\end{align}
where $\max\{|a|,|b|,|ab/q^2|\}<1$. 

By setting $\left(a,b\right)\to\left(0,0\right)$ in \eqref{A14}, we obtain 
\begin{align*}
\sum_{n=0}^{\infty}\frac{\left(-1\right)^nq^{2n^2}}{\left(q^4;q^4\right)_n}=\frac{1}{\left(-q^2;q^2\right)_{\infty}}.
\end{align*}
Then, replacing $q^2$ by $q$ in the above equation and using the relation
\begin{equation*}
\left(q^2;q^2\right)_{\infty}=\left(q^2,q^4,q^6;q^6\right)_{\infty},
\end{equation*}
we obtain \eqref{S.3}.

 Also, by taking $\left(a,b\right)\to\left(0,-q\right)$ in \eqref{A14},  we get \eqref{S.4}.
 \end{proof}
\begin{corollary}
We have
\begin{align}
\sum_{n=0}^{\infty}\frac{\left(-1\right)^nq^{n\left(2n+1\right)}}
{\left(q^2;q^2\right)_n\left(-q;q^2\right)_{n+1}}
&=\left(q,-q^2;q^2\right)_{\infty}
,\tag{S.5}\label{S.5}\\
\sum_{n=0}^{\infty}\frac{q^{n\left(2n+1\right)}}{\left(q;q\right)_{2n+1}}
&=\frac{\left(-q,-q^3,q^4;q^4\right)_{\infty}}{\left(q^2;q^2\right)_{\infty}}
\tag{S.9}\label{S.9}\\
&=\frac{\left(q^2,q^6,q^8;q^8\right)_{\infty}
\left(q^4,q^{12};q^{16}\right)_{\infty}}{\left(q;q\right)_{\infty}}
,\tag{S.84}\label{S.84}\\
\sum_{n=0}^{\infty}\frac{\left(-q;q^2\right)_n}
{\left(q;q\right)_{2n+1}}q^{n\left(n+1\right)}&=
\frac{\left(q^4,q^8,q^{12};q^{12}\right)_{\infty}}{\left(q;q\right)_{\infty}}, \tag{S.51}\label{S.51}
\\
\sum_{n=0}^{\infty}\frac{\left(q^2;q^4\right)_n\left(-q^2;q^2\right)_n
q^{n\left(n+1\right)}}{\left(q^4;q^4\right)_n\left(q;q^2\right)_n
\left(q;q^2\right)_{n+1}}
&=
\frac{\left(-q^2;q^2\right)_{\infty}\left(-q,-q^3,q^4;q^4\right)_{\infty}}
{\left(q^2;q^2\right)_{\infty}},\tag{S.64}\label{S.64}\\
\sum_{n=0}^{\infty}\frac{\left(-q^2;q^2\right)_n}
{\left(q;q\right)_{2n+1}}q^{n^2}
&=\frac{\left(-q;q^2\right)_{\infty}}{\left(q;q^2\right)_{\infty}}.\tag{\cite[Entry 1.7.13]{Ramanujan}}\label{N1}
\end{align}
\end{corollary}
\begin{proof}
Substituting $\alpha=-q,c=-q^3$ into Theorem \ref{thm9} results in
\begin{align}
\sum_{n=0}^{\infty}\frac{\left(q^2/a,q^2/b;q^2\right)_{n}}
{\left(q^2,-q^3;q^2\right)_{n}}
\left(-ab/q\right)^n=\frac{\left(-aq,-bq;q^2\right)_{\infty}}{\left(-q^3,-ab/q;q^2\right)_{\infty}},\label{A2}
\end{align}
where $\max\{|aq|,|bq|,|ab/q|\}<1$. 

After inserting $\left(a,b\right)\to\left(0,0\right)$ into \eqref{A2}, 
we obtain \eqref{S.5}.  Furthermore, replacing $q$ by $-q$ in \eqref{S.5} and utilizing the following relation
\begin{equation*}
\frac{1}{\left(q;q^2\right)_{\infty}}
=\frac{\left(-q,-q^3,q^4;q^4\right)_{\infty}}{\left(q^2;q^2\right)_{\infty}}
=\frac{\left(q^2,q^6,q^8;q^8\right)_{\infty}
\left(q^4,q^{12};q^{16}\right)_{\infty}}{\left(q;q\right)_{\infty}},
\end{equation*}
we  obtain \eqref{S.9} and \eqref{S.84}.

In addition, setting $\left(a,b\right)\to\left(0,q\right)$ into \eqref{A2}, we get
\begin{align}
\sum_{n=0}^{\infty}\frac{\left(q;q^2\right)_n}{\left(q^2,-q^3;q^2\right)_n}q^{n(n+1)}
=\frac{\left(-q^2;q^2\right)_{\infty}}{\left(-q^3;q^2\right)_{\infty}}.\label{C7}
\end{align}
Replacing $q$ by $-q$ in \eqref{C7}, we have
\begin{align*}
\sum_{n=0}^{\infty}\frac{\left(-q;q^2\right)_n}{\left(q;q\right)_{2n+1}}q^{n(n+1)}
=\frac{\left(-q^2;q^2\right)_{\infty}}{\left(q;q^2\right)_{\infty}}.
\end{align*}
Then, combining the above equation with the following two relations
\begin{align*}
&\sum_{n=0}^{\infty}\frac{\left(q^2;q^4\right)_n\left(-q^2;q^2\right)_n}{\left(q^4;q^4\right)_n\left(q;q^2\right)_n
\left(q;q^2\right)_{n+1}}q^{n\left(n+1\right)}
=\sum_{n=0}^{\infty}\frac{\left(-q;q^2\right)_n
}{\left(q;q\right)_{2n+1}}q^{n\left(n+1\right)},\\
&\frac{1}{\left(q;q^2\right)_{\infty}}=\frac{\left(q^2;q^2\right)_{\infty}}{\left(q;q\right)_{\infty}}
=\frac{\left(-q,-q^3,q^4;q^4\right)_{\infty}}{\left(q^2;q^2\right)_{\infty}},
\end{align*}
we arrive at \eqref{S.51} and \eqref{S.64}.

Finally, replacing $q$ by $-q$ in \eqref{A2} and taking $\left(a,b\right)\to\left(0,-1\right)$, we get \ref{N1}.
\end{proof}
%


\begin{corollary}
We have
\begin{align}
\sum_{n=0}^{\infty}\frac{\left(-1;q\right)_{n}}
{\left(q;q\right)_n}q^{\frac{1}{2}n\left(n+1\right)}
&=\frac{\left(-q;q\right)_{\infty}\left(q^2,q^2,q^4;q^4\right)_{\infty}}
{\left(q;q\right)_{\infty}},\tag{S.12}\label{S.12}\\
\sum_{n=0}^{\infty}\frac{\left(-q;q\right)_{n}}
{\left(q;q\right)_n}q^{\frac{1}{2}n\left(n-1\right)}
&=\frac{\left(-q;q\right)_{\infty}}
{\left(q;q\right)_{\infty}}\left(\left(q,q^3,q^4;q^4\right)_{\infty}
+\left(q^2,q^2,q^4;q^4\right)_{\infty}\right),\tag{S.13}\label{S.13}\\
\sum_{n=0}^{\infty}\frac{q^{n^2}}
{\left(q;q\right)_n}
&=\frac{\left(q^2,q^3,q^5;q^5\right)_{\infty}}
{\left(q;q\right)_{\infty}},\tag{S.18}\label{S.18}\\
\sum_{n=0}^{\infty}\frac{\left(-1;q\right)_{2n}}
{\left(q^2;q^2\right)_n}q^{n}
&=\frac{\left(-q;q\right)_{\infty}\left(q^3,q^3,q^6;q^6\right)_{\infty}}
{\left(q;q\right)_{\infty}},\tag{S.24}\label{S.24}\\
\sum_{n=0}^{\infty}\frac{\left(-1\right)^n\left(-1,q;q^2\right)_n}
{\left(q^2;q^2\right)_n}q^{n}
&=\frac{\left(q,-q^2;q^2\right)_{\infty}\left(-q^3,-q^3,q^6;q^6\right)_{\infty}}
{\left(-q,q^2;q^2\right)_{\infty}},\tag{S.30}\label{S.30}\\
\sum_{n=0}^{\infty}\frac{\left(-q;q^2\right)_{n}}
{\left(q^2;q^2\right)_{n}}q^{n^2}
&=\frac{\left(-q^3;q^2\right)_{\infty}\left(q^3,q^5,q^8;q^8\right)_{\infty}}
{\left(q^2;q^2\right)_{\infty}}.\tag{S.36}\label{S.36}
\end{align}
\end{corollary}

\begin{proof}
Taking $\alpha=1,c=0$ into Theorem \ref{thm9}, we deduce that
\begin{align}
&\frac{\left(q^2,ab/q^2;q^2\right)_{\infty}}
{\left(a,b;q^2\right)_{\infty}}\sum_{n=0}^{\infty}
\frac{\left(q^2/a,q^2/b;q^2\right)_{n}}
{\left(q^2;q^2\right)_n}\left(ab/q^2\right)^n\nonumber\\
&\quad=1+\sum_{n=1}^{\infty}\left(-1\right)^n\left(1+q^{2n}\right)
\frac{\left(q^2/a,q^2/b;q^2\right)_{n}}{\left(a,b;q^2\right)_n}
a^nb^nq^{3n^2-3n},\label{q6}
\end{align}
where $\max\{|a|,|b|,|ab/q^2|\}<1$.

When taking $\left(a,b\right)\to\left(0,-q^2\right)$ into \eqref{q6} and employing the Jacobi triple product identity \eqref{Jacobi}
with $\left(q,z\right)\to\left(q^8,-q^4\right)$, we deduce that
\begin{equation}\label{equ8}
\frac{\left(q^2;q^2\right)_{\infty}}{\left(-q^2;q^2\right)_{\infty}}
\sum_{n=0}^{\infty}\frac{\left(-1;q^2\right)_{n}}{\left(q^2;q^2\right)_{n}}
q^{n\left(n+1\right)}
=\left(q^4,q^4,q^8;q^8\right)_{\infty}.
\end{equation}
Then, we complete the proof of \eqref{S.12} by letting $q^2\to q$ in \eqref{equ8}.

Moreover, we take $\left(a,b\right)\to\left(0,-1\right)$ into  \eqref{q6} to get
\begin{align}
&\frac{\left(q^2;q^2\right)_{\infty}}{\left(-q^2;q^2\right)_{\infty}}
\sum_{n=0}^{\infty}\frac{\left(-q^2;q^2\right)_{n}}{\left(q^2;q^2\right)_{n}}
q^{n\left(n-1\right)}\nonumber\\
&\quad=\sum_{n=-\infty}^{\infty}\left(-1\right)^nq^{4n^2-2n}
+\sum_{n=-\infty}^{\infty}\left(-1\right)^nq^{4n^2}\nonumber\\
&\quad=\left(q^2,q^6,q^8;q^8\right)_{\infty}
+\left(q^4,q^4,q^8;q^8\right)_{\infty},\label{equ9}
\end{align}
where the last identity follows from the Jacobi triple product identity \eqref{Jacobi} with $\left(q,z\right)\to\left(q^8,-q^2\right)$ and $\left(q,z\right)\to\left(q^8,-q^4\right)$, respectively.
Then, we arrive at \eqref{S.13} by letting $q^2\to q$ in \eqref{equ9}.

Taking $\left(a,b\right)\to\left(0,0\right)$ into \eqref{q6} and employing the Jacobi triple product identity \eqref{Jacobi} with $\left(q,z\right)\to\left(q^{10},-q^4\right)$, we obtain
\begin{equation}\label{Equ8}
\left(q^2;q^2\right)_{\infty}
\sum_{n=0}^{\infty}\frac{q^{2n^2}}{\left(q^2;q^2\right)_{n}}
=\left(q^4,q^6,q^{10};q^{10}\right)_{\infty}.
\end{equation}
Then, we derive \eqref{S.18} by letting $q^2\to q$ in \eqref{Equ8}.

In addition, taking $\left(a,b\right)\to\left(-q,-q^2\right)$, $\left(a,b\right)\to\left(q,-q^2\right)$, $\left(a,b\right)\to\left(-q,0\right)$, respectively, into  \eqref{q6} and employing the Jacobi triple product identity \eqref{Jacobi}, we obtain \eqref{S.24}, \eqref{S.30} and  \eqref{S.36}.
\end{proof}
\begin{corollary}
We have
\begin{align}
\sum_{n=0}^{\infty}\frac{\left(-q;q^2\right)_n}{\left(q;q\right)_{2n}}q^{n^2}
&=
\frac{\left(-q;q^2\right)_{\infty}\left(-q^2,-q^4,q^6;q^6\right)_{\infty}}
{\left(q^2;q^2\right)_{\infty}},\tag{S.29}\label{S.29}\\
\sum_{n=0}^{\infty}\frac{q^{2n^2}}{\left(q;q\right)_{2n}}
&=
\frac{\left(-q^3,-q^5,q^8;q^8\right)_{\infty}}
{\left(q^2;q^2\right)_{\infty}}.\tag{S.39}\label{S.39}
\end{align}
\end{corollary}
\begin{proof}
Taking $\alpha=1,c=q$ into Theorem \ref{thm9}, we deduce that
\begin{align}
&\frac{\left(q^2,ab/q^2;q^2\right)_{\infty}}
{\left(a,b;q^2\right)_{\infty}}\sum_{n=0}^{\infty}
\frac{\left(q^2/a,q^2/b;q^2\right)_{n}}
{\left(q,q^2;q^2\right)_n}\left(ab/q^2\right)^n\nonumber\\
&\quad=1+\sum_{n=1}^{\infty}\left(1+q^{2n}\right)
\frac{\left(q^2/a,q^2/b;q^2\right)_{n}}{\left(a,b;q^2\right)_n}
a^nb^nq^{2n^2-3n},\label{A15}
\end{align}
where $\max\{|a|,|b|,|ab/q^2|\}<1$.

Substituting $\left(a,b\right)\to\left(0,-q\right)$ into \eqref{A15} and using the Jacobi triple product with $\left(q,z\right)\to\left(q^6,q^2\right)$, we derive \eqref{S.29}.

Additionally, taking $\left(a,b\right)\to\left(0,0\right)$ into  \eqref{A15} and using Jacobi triple product with $\left(q,z\right)\to\left(q^8,q^3\right)$, we obtain \eqref{S.39}.
\end{proof}
\begin{corollary}
We have
\begin{align}
\sum_{n=0}^{\infty}\frac{\left(-q;q^2\right)_{n}}
{\left(q^4;q^4\right)_n}q^{n^2}
&=\frac{\left(-q;q^2\right)_{\infty}\left(q^3,q^3,q^6;q^6\right)_{\infty}}
{\left(q^2;q^2\right)_{\infty}}.\tag{S.25}\label{S.25}
\end{align}
\end{corollary}
\begin{proof}
Taking $\alpha=1,c=-q^2$ into Theorem \ref{thm9}, we deduce that
\begin{align}
&\frac{\left(q^2,ab/q^2;q^2\right)_{\infty}}
{\left(a,b;q^2\right)_{\infty}}\sum_{n=0}^{\infty}
\frac{\left(q^2/a,q^2/b;q^2\right)_{n}}
{\left(q^2,-q^2;q^2\right)_n}\left(ab/q^2\right)^n\nonumber\\
&\quad=1+2\sum_{n=1}^{\infty}\left(-1\right)^n
\frac{\left(q^2/a,q^2/b;q^2\right)_{n}}{\left(a,b;q^2\right)_n}
a^nb^nq^{2n^2-2n},\label{q61}
\end{align}
where $\max\{|a|,|b|,|ab/q^2|\}<1$.

Inserting $\left(a,b\right)\to\left(0,-q\right)$ into \eqref{q61} and employing the Jacobi  triple product identity \eqref{Jacobi} with $\left(q,z\right)\to\left(q^6,-q^3\right)$, we get \eqref{S.25}.
\end{proof}
\begin{corollary}
We have
\begin{align}
\sum_{n=0}^{\infty}\frac{\left(-q;q\right)_n
}{\left(q;q\right)_n}q^{\frac{1}{2}n\left(n+1\right)}
&=
\frac{\left(-q;q\right)_{\infty}\left(q,q^3,q^4;q^4\right)_{\infty}}
{\left(q;q\right)_{\infty}},\tag{S.8}\label{S.8}\\
\sum_{n=0}^{\infty}\frac{q^{n\left(n+1\right)}}
{\left(q;q\right)_n}
&=\frac{\left(q,q^4,q^5;q^5\right)_{\infty}}
{\left(q;q\right)_{\infty}},\tag{S.14}\label{S.14}\\
\sum_{n=0}^{\infty}\frac{\left(-q;q^2\right)_n}
{\left(q^2;q^2\right)_n}q^{n\left(n+2\right)}
&=\frac{\left(-q;q^2\right)_{\infty}\left(q,q^7,q^8;q^8\right)_{\infty}}
{\left(q^2;q^2\right)_{\infty}}.\tag{S.34} \label{S.34}
\end{align}
\end{corollary}
\begin{proof}
By taking $\alpha=q^2,c=0$ into Theorem \ref{thm9}, we derive 
\begin{align}
&\frac{\left(q^2,ab;q^2\right)_{\infty}}
{\left(aq^2,bq^2;q^2\right)_{\infty}}\sum_{n=0}^{\infty}
\frac{\left(q^2/a,q^2/b;q^2\right)_{n}}
{\left(q^2;q^2\right)_{n}}a^nb^n\nonumber\\
&\quad=\sum_{n=0}^{\infty}\left(1-q^{4n+2}\right)
\frac{\left(q^2/a,q^2/b;q^2\right)_{n}}
{\left(aq^2,bq^2;q^2\right)_n}
a^nb^nq^{3n^2+n},\label{A17}
\end{align}
where $\max\{|ab|,|aq^2|,|bq^2|\}<1$.

Taking $\left(a,b\right)\to\left(0,-1\right)$ into \eqref{A17} and employing the Jacobi triple product identity \eqref{Jacobi} with $\left(q,z\right)\to\left(q^8,-q^2\right)$, we obtain
\begin{equation}\label{equa8}
\frac{\left(q^2;q^2\right)_{\infty}}{\left(-q^2;q^2\right)_{\infty}}
\sum_{n=0}^{\infty}\frac{\left(-q^2;q^2\right)_{n}}{\left(q^2;q^2\right)_{n}}
q^{n\left(n+1\right)}
=\left(q^2,q^6,q^8;q^8\right)_{\infty}.
\end{equation}
Then, we derive \eqref{S.8} by letting $q^2\to q$ in \eqref{equa8}.

In addition, taking $\left(a,b\right)\to\left(0,0\right)$ into  \eqref{A17} and applying the Jacobi triple product identity \eqref{Jacobi} with $\left(q,z\right)\to\left(q^{10},-q^2\right)$, we obtain
\begin{equation}\label{equa7}
\left(q^2;q^2\right)_{\infty}
\sum_{n=0}^{\infty}\frac{q^{2n\left(n+1\right)}}{\left(q^2;q^2\right)_{n}}
=\left(q^2,q^8,q^{10};q^{10}\right)_{\infty}.
\end{equation}
Then, we complete the proof of \eqref{S.14} by letting $q^2\to q$ in \eqref{equa7}.

Taking $\left(a,b\right)\to\left(0,-q\right)$ into \eqref{A17} and utilizing the Jacobi triple product identity \eqref{Jacobi} with $\left(q,z\right)\to\left(q^8,-q\right)$, we immediately obtain \eqref{S.34}.
\end{proof}
\begin{corollary}
We have
\begin{align}
\sum_{n=0}^{\infty}\frac{q^{\frac{n\left(n+1\right)}{2}}}
{\left(q;q\right)_{n}}
&=\left(-q;q\right)_{\infty},\tag{S.2}\label{S.2}\\
\sum_{n=0}^{\infty}\frac{
q^{n\left(n+1\right)}}{\left(q^2;q^2\right)_n}
&=
\frac{\left(q,q^3,q^4;q^4\right)_{\infty}}
{\left(q;q\right)_{\infty}},\tag{S.7}\label{S.7}
\\
\sum_{n=0}^{\infty}\frac{\left(-q;q^2\right)_n}
{\left(q^4;q^4\right)_n}q^{n\left(n+2\right)}
&=\frac{\left(-q;q^2\right)_{\infty}\left(q,q^5,q^6;q^6\right)_{\infty}}
{\left(q^2;q^2\right)_{\infty}}.\tag{\cite[Entry 5.3.7]{Ramanujan}}\label{new2}
\end{align}
\end{corollary}
\begin{proof}
By inserting $\alpha=q^2,c=-q^2$ into Theorem \ref{thm9}, we find that 
\begin{align}
&\frac{\left(q^2,ab;q^2\right)_{\infty}}
{\left(aq^2,bq^2;q^2\right)_{\infty}}\sum_{n=0}^{\infty}
\frac{\left(q^2/a,q^2/b;q^2\right)_{n}}
{\left(q^2,-q^2;q^2\right)_{n}}a^nb^n\nonumber\\
&\quad=\sum_{n=0}^{\infty}\left(-1\right)^n\left(1-q^{4n+2}\right)
\frac{\left(q^2/a,q^2/b;q^2\right)_{n}}
{\left(aq^2,bq^2;q^2\right)_n}
a^nb^nq^{2n^2},\label{A16}
\end{align}
where $\max\{|ab|,|aq^2|,|bq^2|\}<1$.

Letting $\left(a,b\right)\to\left(0,0\right)$ in \eqref{A16} and utilizing the Jacobi triple product identity \eqref{Jacobi} with $\left(q,z\right)\to\left(q^8,-q^2\right)$, we deduce that
\begin{equation*}
\sum_{n=0}^{\infty}\frac{q^{2n\left(n+1\right)}}{\left(q^4;q^4\right)_n}
=\frac{\left(q^2,q^6,q^8;q^8\right)_{\infty}}{\left(q^2;q^2\right)_{\infty}}=\left(-q^4;q^4\right)_{\infty},
\end{equation*}
which can be reduced to \eqref{S.2} and  \eqref{S.7} by taking $q^4\to q$ and $q^2\to q$, respectively.

In addition, taking $\left(a,b\right)\to\left(0,-q\right)$ into  \eqref{A16}, we deduce that
\begin{align*}
\frac{\left(q^2;q^2\right)_{\infty}}{\left(-q;q^2\right)_{\infty}}
\sum_{n=0}^{\infty}
\frac{\left(-q;q^2\right)_{n}}
{\left(q^2,-q^2;q^2\right)_{n}}q^{n\left(n+2\right)}
=\sum_{n=0}^{\infty}\left(-1\right)^n
\left(1-q^{2n+1}\right)q^{3n^2+2n}
=\left(q,q^5,q^6;q^6\right)_{\infty},
\end{align*}
where the last equation follows from  the Jacobi triple product identity \eqref{Jacobi} with $\left(q,z\right)\to\left(q^6,-q\right)$. Hence, we complete the proof of \ref{new2}.
\end{proof}
\begin{corollary}
We have
\begin{align}
\sum_{n=0}^{\infty}\frac{\left(-q^2;q^2\right)_n}
{\left(q;q\right)_{2n+1}}q^{n\left(n+1\right)}
&=\frac{\left(-q^2;q^2\right)_{\infty}\left(-q,-q^5,q^6;q^6\right)_{\infty}}
{\left(q^2;q^2\right)_{\infty}},\tag{S.28}\label{S.28}\\
\sum_{n=0}^{\infty}\frac{q^{2n\left(n+1\right)}}
{\left(q;q\right)_{2n+1}}
&=\frac{\left(-q,-q^7,q^8;q^8\right)_{\infty}}
{\left(q^2;q^2\right)_{\infty}}\tag{S.38} \label{S.38}\\
&=\frac{\left(q^3,q^5,q^8;q^8\right)_{\infty}
\left(q^2,q^{14};q^{16}\right)_{\infty}}{\left(q;q\right)_{\infty}}, \tag{S.86}\label{S.86}\\
\sum_{n=0}^{\infty}\frac{\left(-q;q^2\right)_n
}{\left(q;q\right)_{2n+1}}q^{n\left(n+2\right)}
&=
\frac{\left(q^2,q^{10},q^{12};q^{12}\right)_{\infty}}
{\left(q;q\right)_{\infty}}.\tag{S.50}\label{S.50}
\end{align}
\end{corollary}
\begin{proof}
By taking $\alpha=q^2,c=q^3$ into Theorem \ref{thm9}, we derive 
\begin{align}
&\frac{\left(q^2,ab;q^2\right)_{\infty}}
{\left(aq^2,bq^2;q^2\right)_{\infty}}\sum_{n=0}^{\infty}
\frac{\left(q^2/a,q^2/b;q^2\right)_{n}}
{\left(q;q\right)_{2n+1}}a^nb^n\nonumber\\
&\quad=\sum_{n=0}^{\infty}\left(1+q^{2n+1}\right)
\frac{\left(q^2/a,q^2/b;q^2\right)_{n}}
{\left(aq^2,bq^2;q^2\right)_n}
a^nb^nq^{2n^2+n},\label{q7}
\end{align}
where $\max\{|ab|,|aq^2|,|bq^2|\}<1$.

Inserting $\left(a,b\right)\to\left(0,-1\right)$ into \eqref{q7} and employing the Jacobi triple product identity \eqref{Jacobi} with $\left(q,z\right)\to\left(q^6,q\right)$, we obtain \eqref{S.28}.

Taking $\left(a,b\right)\to\left(0,0\right)$ into \eqref{q7} and employing the Jacobi triple product identity \eqref{Jacobi} with $\left(q,z\right)\to\left(q^8,q\right)$, we obtain
\begin{equation*}
\left(q^2;q^2\right)_{\infty}
\sum_{n=0}^{\infty}\frac{q^{n\left(n+2\right)}}{\left(q;q\right)_{2n+1}}
=\left(-q,-q^7,q^8;q^8\right)_{\infty}
=\frac{\left(q^3,q^5,q^8;q^8\right)_{\infty}
\left(q^2,q^{14};q^{16}\right)_{\infty}}
{\left(q;q^2\right)_{\infty}},
\end{equation*}
which means \eqref{S.38} and \eqref{S.86} are true.

Taking $\left(a,b\right)\to\left(0,-q\right)$ into \eqref{q7} and employing the Jacobi triple product identity \eqref{Jacobi} with $\left(q,z\right)\to\left(q^6,q^6\right)$, we obtain
\begin{equation*}
\frac{\left(q^2;q^2\right)_{\infty}}{\left(-q;q^2\right)_{\infty}}
\sum_{n=0}^{\infty}\frac{\left(-q;q^2\right)_{n}}{\left(q;q\right)_{2n+1}}
q^{n\left(n+2\right)}
=\sum_{n=0}^{\infty}q^{3n^2+3n}
=\left(-q^6,-q^6,q^{6};q^{6}\right)_{\infty},
\end{equation*}
which implies that \eqref{S.50} is true.
\end{proof}
%
\begin{corollary}
We have
\begin{align}
\sum_{n=0}^{\infty}\frac{\left(-1;q^2\right)_n}
{\left(q;q\right)_{2n}}q^{n^2}&=
\frac{\left(-q;q^2\right)_{\infty}\left(q^2,q^6;q^{8}\right)_{\infty}
\left(q^4;q^4\right)_{\infty}}
{\left(q;q\right)_{\infty}},
\tag{S.47}\label{S.47}\\
\sum_{n=0}^{\infty}\frac{\left(-q;q^2\right)_n}
{\left(q^2;q^4\right)_{n}\left(q^2;q^2\right)_{n}}q^{n\left(2n-1\right)}
&=
\frac{\left(-q;q^2\right)_{\infty}\left(q^4,q^8,q^{12};q^{12}\right)_{\infty}}
{\left(q^2;q^2\right)_{\infty}}.
\tag{S.52}\label{S.52}
\end{align}
\end{corollary}
\begin{proof}
We make the substitution $\alpha=1/q, c=q$ into Theorem \ref{thm9}
to obtain
\begin{align}
\sum_{n=0}^{\infty}\frac{\left(q^2/a,q^2/b;q^2\right)_{n}}
{\left(q,q^2;q^2\right)_n}\left(ab/q^3\right)^n
=\frac{\left(a/q,b/q;q^2\right)_{\infty}}
{\left(q,ab/q^3;q^2\right)_{\infty}},\label{q8}
\end{align}
where $\max\{|a/q|, |b/q|, |ab/q^3|\}<1$.

We observe that
\begin{equation*}
\frac{\left(-q;q^2\right)_{\infty}}{\left(q;q^2\right)_{\infty}}=
\frac{\left(-q;q^2\right)_{\infty}\left(q^2,q^6;q^{8}\right)_{\infty}
\left(q^4;q^4\right)_{\infty}}
{\left(q;q\right)_{\infty}}.
\end{equation*}
Then, by substituting $\left(a,b\right)\to\left(0,-q^2\right)$ into \eqref{q8}, we get \eqref{S.47}.

Similarly, in \eqref{q8},  when we set $\left(a,b\right)\to\left(0,0\right)$,  the result simplifies to  \eqref{S.52}.
\end{proof}
Next, with the help of a ${_2\phi_1}$ summation formula found in \cite[Eq. (6.125)]{Harsh}:
\begin{align}
&{_2\phi_1}\left(
\begin{array}{cc}
a,\,b\\
aq^2/b
\end{array};q,\,-q/b\right)=\frac{\left(-q;q\right)_{\infty}\left(
\left(aq,aq^2/b^2;q^2\right)_{\infty}-q/b
\left(a,aq^3/b^2;q^2\right)_{\infty}\right)}{\left(1-q/b\right)
\left(aq^2/b,-q/b;q\right)_{\infty}},\label{C1}
\end{align} 
we establish the following transformation formula. 
\begin{theorem}\label{T7}
For $\max\{|a^2|,|aq^2|,|bq^2|,|ab|\}<1$, we have
\begin{align*}
&\frac{\left(q^4,ab;q^2\right)_{\infty}}
{\left(-q^2,a^2,bq^2;q^2\right)_{\infty}\left(q^2;q^4\right)_{\infty}}
\sum_{n=0}^{\infty}\frac{\left(q^2/a,q^2/b;q^2\right)_n}
{\left(q^2,q^3,-q^3;q^2\right)_n}a^nb^n\\
&\quad =\sum_{n=0}^{\infty}\left(-1\right)^n
\frac{\left(q^2/a,q^2/b;q^2\right)_{2n}}{\left(aq^2,bq^2;q^2\right)_{2n}}
a^{2n}b^{2n}q^{6n^2+2n}\\
&\qquad-
\sum_{n=0}^{\infty}\left(-1\right)^n
\frac{\left(q^2/a,q^2/b;q^2\right)_{2n+1}}{\left(aq^2,bq^2;q^2\right)_{2n+1}}
a^{2n+1}b^{2n+1}q^{6n^2+10n+4}.
\end{align*}
\end{theorem}
\begin{proof}
First, replacing $q$ by $q^2$ and substituting $\left(a,b\right)\to\left(q^{-2n},q^{1-2n}\right)$ into \eqref{C1}, we have
\begin{align}
&{_2\phi_1}\left(
\begin{array}{cc}
q^{-2n},\,q^{1-2n}\\
q^3
\end{array};q^2,\,-q^{2n+1}\right)\nonumber\\
&\quad=\frac{\left(-q^2;q^2\right)_{\infty}\left(
\left(q^{2-2n},q^{2+2n};q^4\right)_{\infty}-q^{1+2n}
\left(q^{-2n},q^{4+2n};q^4\right)_{\infty}\right)}{\left(1-q^{2n+1}\right)
\left(q^3,-q^{2n+1};q^2\right)_{\infty}}.\label{eq1}
\end{align}
On the other hand, from \cite[p.71]{Gasper_and_Rahman}, we find 
\begin{equation}\label{C2}
_3\phi_2\left(
\begin{array}{ccc}
q^{-n},\,b,\,0\\
d,\,e
\end{array};q,\,q\right)
=\frac{\left(-e\right)^nq^{{\frac{n\left(n-1\right)}{2}}}}{\left(e;q\right)_n}
{_2\phi_1}\left(
\begin{array}{cc}
q^{-n},\,d/b\\
d
\end{array};q,\,bq/e\right).
\end{equation}
Thereafter, setting $\left(q,b,d,e\right)\to\left(q^2,q^{2n+2},q^3,-q^3\right)$ in \eqref{C2} and applying \eqref{eq1}, we get
\begin{align}
&{_3\phi_2}\left(
\begin{array}{ccc}
q^{-2n},\,q^{2n+2},\,0\\
q^3,-q^3
\end{array};q^2,\,q^2\right)\nonumber\\
&\quad=\frac{q^{n^2+2n}}{\left(-q^3;q^2\right)_n}{_2\phi_1}\left(
\begin{array}{cc}
q^{-2n},\,q^{1-2n}\\
q^3
\end{array};q^2,\,-q^{2n+1}\right)\nonumber\\
&\quad=\frac{\left(1-q^2\right)\left(-q^2;q^2\right)_{\infty}}
{\left(1-q^{4n+2}\right)\left(q^2;q^4\right)_{\infty}}q^{n^2+2n}\left(
\left(q^{2-2n},q^{2+2n};q^4\right)_{\infty}-q^{1+2n}
\left(q^{-2n},q^{4+2n};q^4\right)_{\infty}\right).\label{eq2}
\end{align}
Then, by inserting $\left(q,\alpha,\beta,c,d\right)\to\left(q^2,q^2,0,q^3,-q^3\right)$ into \eqref{Formula1}, we deduce that
\begin{align*}
&\frac{\left(q^2;q^4\right)_{\infty}\left(q^4,ab;q^2\right)_{\infty}}
{\left(-q^2,a^2,bq^2;q^2\right)_{\infty}}
\sum_{n=0}^{\infty}\frac{\left(q^2/a,q^2/b;q^2\right)_n}
{\left(q^2,q^3,-q^3;q^2\right)_n}a^nb^n\\
&\quad=\sum_{n=0}^{\infty}\left(-1\right)^n
\frac{\left(q^2/a,q^2/b;q^2\right)_n}{\left(aq^2,bq^2;q^2\right)_n}a^nb^nq^{2n^2+n}
\left(
\left(q^{2-2n},q^{2+2n};q^4\right)_{\infty}-q^{1+2n}
\left(q^{-2n},q^{4+2n};q^4\right)_{\infty}\right)\\
&\quad=\left(q^2;q^4\right)^2_{\infty}\left(\sum_{n=0}^{\infty}\left(-1\right)^n
\frac{\left(q^2/a,q^2/b;q^2\right)_{2n}}{\left(aq^2,bq^2;q^2\right)_{2n}}
a^{2n}b^{2n}q^{6n^2+2n}\right.\\
&\qquad\left.-
\sum_{n=0}^{\infty}\left(-1\right)^n
\frac{\left(q^2/a,q^2/b;q^2\right)_{2n+1}}{\left(aq^2,bq^2;q^2\right)_{2n+1}}
a^{2n+1}b^{2n+1}q^{6n^2+10n+4}\right),
\end{align*}
which is what we want to show.
\end{proof}
\begin{corollary}
We have
\begin{align}
\sum_{n=0}^{\infty}\frac{\left(-q;q\right)_nq^{\frac{n^2+n}{2}}}{\left(q;q^2\right)_{n+1}
\left(q;q\right)_n}
&=\frac{\left(-q;q\right)_{\infty}\left(q^3,q^{7},q^{10};q^{10}\right)_{\infty}}{\left(q;q\right)_{\infty}}
,\tag{S.45}\label{S.45}\\
\sum_{n=0}^{\infty}\frac{q^{n^2+n}}{\left(q;q^2\right)_{n+1}\left(q;q\right)_n}
&=\frac{\left(q^4,q^{10},q^{14};q^{14}\right)_{\infty}}{\left(q;q\right)_{\infty}}
,\tag{S.60}\label{S.60}\\
\sum_{n=0}^{\infty}\frac{q^{n^2}}{\left(q;q\right)_{2n+1}}&=
\frac{\left(q^8,q^{12},q^{20};q^{20}\right)_{\infty}+q\left(q^4,q^{16},q^{20};q^{20}\right)_{\infty}}{\left(q;q\right)_{\infty}\left(-q^2;q^2\right)_{\infty}}.\label{A21}
\end{align}
\end{corollary}
\begin{proof}
Taking $\left(a,b\right)\to\left(0,-1\right)$ into Theorem \ref{T7} and using the Jacobi triple product identity \eqref{Jacobi} with $\left(q,z\right)\to\left(q^{20},-q^6\right)$, we derive
\begin{align*}
\frac{\left(q^4;q^2\right)_{\infty}}
{\left(q^2;q^4\right)_{\infty}\left(-q^2;q^2\right)^2_{\infty}}
\sum_{n=0}^{\infty}\frac{\left(-q^2;q^2\right)_nq^{n^2+n}}
{\left(q^2,q^3,-q^3;q^2\right)_{n}}
=\sum_{n=-\infty}^{\infty}\left(-1\right)^nq^{10n^2+4n}
=\left(q^6,q^{14},q^{20};q^{20}\right)_{\infty}.
\end{align*}
Then, replacing $q^2$ by $q$, we arrive at \eqref{S.45}.

Inserting $\left(a,b\right)\to\left(0,0\right)$ into Theorem \ref{T7} and applying the Jacobi triple product identity \eqref{Jacobi} with $\left(q,z\right)\to\left(q^{28},-q^8\right)$, we obtain
\begin{align*}
\frac{\left(q^4;q^2\right)_{\infty}}
{\left(q^2;q^4\right)_{\infty}\left(-q^2;q^2\right)_{\infty}}
\sum_{n=0}^{\infty}\frac{q^{2n^2+2n}}{\left(q^2,q^3,-q^3;q^2\right)_{n}}
=\sum_{n=-\infty}^{\infty}\left(-1\right)^nq^{14n^2+6n}
=\left(q^8,q^{20},q^{28};q^{28}\right)_{\infty}.
\end{align*}
Thereby, we prove \eqref{S.60} is true by replacing $q^2$ by $q$ in the above identity.

In addition, letting $\left(a,b\right)\to\left(0,-1/q\right)$ into Theorem \ref{T7}, we get
\begin{align*}
&\frac{\left(q^4;q^2\right)_{\infty}}
{\left(-q^2,-q^3;q^4\right)_{\infty}\left(q^2;q^4\right)_{\infty}}
\sum_{n=0}^{\infty}\frac{q^{n^2}}{\left(q^2,q^3;q^2\right)_{n}}\\
&\quad=\sum_{n=-\infty}^{\infty}\left(-1\right)^nq^{10n^2+2n}+q\sum_{n=-\infty}^{\infty}\left(-1\right)^nq^{10n^2+6n}\\
&\quad=\left(q^8,q^{12},q^{20};q^{20}\right)_{\infty}+q\left(q^4,q^{16},q^{20};q^{20}\right)_{\infty},
\end{align*}
which reduces to \eqref{A21} after simplification.
\end{proof}
Another interesting  $_2\phi_1$ summation formula can also be  found in \cite[Eq. (6.110)]{Harsh} as follows,
\begin{align}
&{_2\phi_1}\left(
\begin{array}{cc}
a,\,bq\\
aq/b
\end{array};q,\,-q/b\right)=\frac{b\left(-q;q\right)_{\infty}\left(
\left(a,aq/b^2;q^2\right)_{\infty}-
\left(aq,a/b^2;q^2\right)_{\infty}\right)}{a\left(1-b\right)
\left(aq/b,-1/b;q\right)_{\infty}}.\label{C3}
\end{align}
Now we employ \eqref{C3} to deduce the following transformation formula.
\begin{theorem}\label{T9}
For $\max\left\{|aq^2|,|bq^2|,|ab|\right\}<1$, we have
\begin{align*}
&\frac{\left(q^2,ab;q^2\right)_{\infty}}{\left(aq^2,bq^2;q^2\right)_{\infty}}
\sum_{n=0}^{\infty}\frac{\left(q^2/a,q^2/b;q^2\right)_n}
{\left(-q,q^2;q^2\right)_n\left(q;q^2\right)_{n+1}}a^nb^n\\
&\quad=\sum_{n=0}^{\infty}\left(-1\right)^n\left(1+q^{4n+1}\right)
\frac{\left(q^2/a,q^2/b;q^2\right)_{2n}}{\left(aq^2,bq^2;q^2\right)_{2n}}
\left(ab\right)^{2n}q^{6n^2+2n}\\
&\qquad -\sum_{n=0}^{\infty}\left(-1\right)^n\left(1+q^{4n+3}\right)
\frac{\left(q^2/a,q^2/b;q^2\right)_{2n+1}}{\left(aq^2,bq^2;q^2\right)_{2n+1}}
\left(ab\right)^{2n+1}q^{6n^2+6n+1}
.
\end{align*}
\end{theorem}
\begin{proof}
Replacing $q$ by $q^2$ and setting $a=q^{-2n}, b=q^{-1-2n}$ in \eqref{C3}, we have
\begin{align*}
&{_2\phi_1}\left(\begin{array}{cc}
q^{-2n},\,q^{1-2n}\\
q^3
\end{array};q^2,\,-q^{2n+3}
\right)\\
&\quad=\frac{\left(-q^2;q^2\right)_{\infty}q^{2n}}{\left(1-q^{2n+1}\right)
\left(q^3,-q^{1+2n};q^2\right)_{\infty}}
\left(\left(q^{2n+2},q^{2-2n};q^4\right)_{\infty}
-\left(q^{-2n},q^{4+2n};q^4\right)_{\infty}\right).
\end{align*}
Based on the above identity, we substitute $\left(q,b,d,e\right)\to\left(q^2,q^{2n+2},q^3,-q\right)$ into \eqref{C2} to obtain
\begin{align*}
&{_3\phi_2}\left(\begin{array}{ccc}
q^{-2n},\,q^{2n+2},\,0\\
q^3,\,-q
\end{array};q^2,\,q^2
\right)\\
&\quad=\frac{q^{n^2}}{\left(-q;q^2\right)_n}
{_2\phi_1}\left(\begin{array}{cc}
q^{-2n},\,q^{1-2n}\\
q^3
\end{array};q^2,\,-q^{2n+3}
\right)\\
&\quad=\frac{\left(-q^2;q^2\right)_{\infty}q^{n^2+2n}}{\left(1-q^{2n+1}\right)
\left(q^3,-q;q^2\right)_{\infty}}
\left(\left(q^{2n+2},q^{2-2n};q^4\right)_{\infty}
-\left(q^{-2n},q^{4+2n};q^4\right)_{\infty}\right).
\end{align*}
Therefore, taking $\left(q,\alpha,\beta,c,d\right)\to\left(q^2,q^2,0,q^3,-q\right)$ into \eqref{Formula1}, we have
\begin{align*}
&\frac{\left(q^2,ab;q^2\right)_{\infty}}{\left(aq^2,bq^2;q^2\right)_{\infty}}
\sum_{n=0}^{\infty}\frac{\left(q^2/a,q^2/b;q^2\right)_n}
{\left(-q,q^2,q^3;q^2\right)_n}a^nb^n\\
&\quad=
\sum_{n=0}^{\infty}\left(-1\right)^n\left(1-q^{4n+2}\right)
\frac{\left(q^2/a,q^2/b;q^2\right)_n}{\left(aq^2,bq^2;q^2\right)_n}a^nb^nq^{n^2-n}
{_3\phi_2}\left(\begin{array}{ccc}
q^{-2n},\,q^{2n+2},\,0\\
q^3,\,-q
\end{array};q^2,\,q^2
\right)\\
&\quad=\frac{\left(-q^2;q^2\right)_{\infty}}{\left(q^3,-q;q^2\right)_{\infty}}
\sum_{n=0}^{\infty}\left(-1\right)^n\left(1+q^{2n+1}\right)
\frac{\left(q^2/a,q^2/b;q^2\right)_n}{\left(aq^2,bq^2;q^2\right)_n}
a^nb^nq^{2n^2+n}\\
&\qquad\times\left(\left(q^{2n+2},q^{2-2n};q^4\right)_{\infty}
-\left(q^{-2n},q^{4+2n};q^4\right)_{\infty}\right)\\
&\quad=\left(1-q\right)\left(-q^2;q^2\right)_{\infty}\left(q^2;q^4\right)_{\infty}
\left(
\sum_{n=0}^{\infty}\left(-1\right)^n\left(1+q^{4n+1}\right)
\frac{\left(q^2/a,q^2/b;q^2\right)_{2n}}{\left(aq^2,bq^2;q^2\right)_{2n}}
\left(ab\right)^{2n}q^{6n^2+2n}\right.\\
&\qquad\left. -
\sum_{n=0}^{\infty}\left(-1\right)^n\left(1+q^{4n+3}\right)
\frac{\left(q^2/a,q^2/b;q^2\right)_{2n+1}}{\left(aq^2,bq^2;q^2\right)_{2n+1}}
\left(ab\right)^{2n+1}q^{6n^2+6n+1}
\right)
\end{align*}
as desired.
\end{proof}
\begin{corollary}
We have
\begin{align}
\sum_{n=0}^{\infty}\frac{q^{n^2+2n}}{\left(q;q\right)_{2n+1}}
&=\frac{\left(q^4,q^6,q^{10};q^{10}\right)_{\infty}
\left(q^2,q^{18};q^{20}\right)_{\infty}}{\left(q;q\right)_{\infty}}, \tag{S.96}\label{S.96}\\
\sum_{n=0}^{\infty}\frac{q^{2n^2+2n}}{\left(-q;q^2\right)_{n}\left(q;q\right)_{2n+1}}
&=\frac{\left(q^8,q^{20},q^{28};q^{28}\right)_{\infty}+
q
\left(q^4,q^{24},q^{28};q^{28}\right)_{\infty}
}{\left(q^2;q^2\right)_{\infty}}
.\label{new7}
\end{align}
\end{corollary}
\begin{proof}
By taking $\left(a,b\right)\to\left(0,-q\right)$ into Theorem \ref{T9} and using the Jacobi triple product identity \eqref{Jacobi} with $\left(q,z\right)\to\left(q^{20},-q^4\right)$, we can prove \eqref{S.96} as follows,
\begin{align*}
\sum_{n=0}^{\infty}\frac{q^{n^2+2n}}{\left(q;q\right)_{2n+1}}
=\frac{\left(q^2;q^4\right)_{\infty}\left(q^4,q^{16},q^{20};q^{20}
\right)_{\infty}}{\left(q;q\right)_{\infty}}
=
\frac{\left(q^4,q^6,q^{10};q^{10}\right)_{\infty}
\left(q^2,q^{18};q^{20}\right)_{\infty}}{\left(q;q\right)_{\infty}}.
\end{align*}

Taking $\left(a,b\right)\to\left(0,0\right)$ into Theorem \ref{T9}, we have
\begin{align*}
&\left(q^2;q^2\right)_{\infty}\sum_{n=0}^{\infty}\frac{q^{2n^2+2n}}
{\left(-q;q^2\right)_n\left(q;q\right)_{2n+1}}\nonumber\\
&\quad=
-\sum_{n=0}^{\infty}\left(-1\right)^n\left(1+q^{4n+3}\right)q^{14n^2+18n+5}
+\sum_{n=0}^{\infty}\left(-1\right)^n\left(1+q^{4n+1}\right)q^{14n^2+6n}
\nonumber\\
&\quad=
\left(q^8,q^{20},q^{28};q^{28}\right)_{\infty}
+q\left(q^4,q^{24},q^{28};q^{28}\right)_{\infty}
,
\end{align*}
which implies \eqref{new7} is true.
 \end{proof}

\section{Identities from the $q$-Pfaff-Saalsch\"utz summation}\label{sec3}
Recall the $q$-Pfaff-Saalsch\"utz summation formula \cite[Eq. (1.7.2)]{Gasper_and_Rahman}
\begin{equation}\label{Pfaff}
_3\phi_2\left(
\begin{array}{ccc}
q^{-n},\,\alpha q^n,\,\alpha bc/q\\
\alpha b,\,\alpha c
\end{array};q,\,q\right)=\frac{\left(q/b,q/c;q\right)_{n}}
{\left(\alpha b,\alpha c;q\right)_n}\left(\alpha bc/{q}\right)^n.
\end{equation}
In \cite{Liu2}, Liu applied the $q$-Pfaff-Saalsch\"utz summation formula to sum the $_3\phi_2$ series (with $\beta=\alpha cd/q$)  on the right-hand side of \eqref{Formula1} to obtain
\begin{equation*}
  \frac{\left(q/c,q/d;q\right)_n}{\left(\alpha c,\alpha d;q\right)_n}\left(\alpha cd/q\right)^n,
\end{equation*}
and made \eqref{Formula1} reduce to the following beautiful result.
\begin{theorem}\label{thm2}
For $|\alpha ab/q|<1$, we have
\begin{align}
&\frac{\left(\alpha q,\alpha ab/q;q\right)_{\infty}}{\left(\alpha a,\alpha b;q\right)_{\infty}}{_3\phi_2}\left(
\begin{array}{ccc}
q/a,q/b,\alpha cd/q\\
\alpha c,\alpha d
\end{array};q,{\alpha ab}/{q}\right)\nonumber\\
&\quad=
\sum_{n=0}^{\infty}\frac{\left(1-\alpha q^{2n}\right)
\left(\alpha,q/a,q/b,q/c,q/d;q\right)_n}{\left(1-\alpha\right)\left(q,\alpha a,\alpha b,\alpha c,\alpha d;q\right)_n}
\left(-\alpha^2abcd\right)^n
q^{n\left(n-5\right)/2}
.\label{B4}
\end{align}
\end{theorem}
In this section, by specializing the choices of $\left(q,\alpha,a,b,c,d\right)$ in Theorem \ref{thm2}, we  
find some Rogers-Ramanujan type identities.
\begin{corollary}
We have
\begin{align}
\sum_{n=0}^{\infty}\frac{\left(-q;q\right)_{2n}q^{n\left(n+1\right)}}
{\left(q;q^2\right)_{n+1}\left(q^4;q^4\right)_{n}}
&=\frac{\left(-q^2;q^2\right)_{\infty}\left(-q,-q^3,q^4;q^4\right)_{\infty}}
{\left(q^2;q^2\right)_{\infty}},\tag{S.11}\label{S.11}\\
\sum_{n=0}^{\infty}\frac{\left(q^2;q^4\right)_nq^{2n\left(n+1\right)}}
{\left(q^4;q^4\right)_{n}\left(q;q^2\right)_{n}\left(q;q^2\right)_{n+1}}
&=\frac{\left(-q,-q^5,q^6;q^6\right)_{\infty}}
{\left(q^2;q^2\right)_{\infty}},\tag{S.27\, and\, S.87}\label{S.87}\\
\sum_{n=0}^{\infty}\frac{\left(-q,-q;q^2\right)_n}
{\left(q;q\right)_{2n+1}\left(-q^2;q^2\right)_{n}}q^{n\left(n+2\right)}
&=\left(-q;q\right)_{\infty}\left(-q^4,-q^4;q^4\right)_{\infty}
.\label{new3}
\end{align}
\end{corollary}
\begin{remark}
\eqref{new3} is a special case of Andrews' $q$-analogue of Gauss's second theorem \cite[Eq. (1.9)]{Andrews_1973} 
\begin{align*}
\sum_{k=0}^{\infty}\frac{\left(b,q/b;q\right)_k}{\left(q^2;q^2\right)_k\left(c;q\right)_k}c^kq^{\frac{k^2-k}{2}}=
\frac{\left(bc,qc/b;q^2\right)_{\infty}}{\left(c;q\right)_{\infty}}
\end{align*}
with $q\to q^2, b=-q$ and $c=q^3$. 
\end{remark}
\begin{proof}
By substituting $\left(q,\alpha,c,d\right)\to\left(q^2,q^2,-1,q\right)$ into \eqref{B4}, we obtain
\begin{align}
&\frac{\left(q^2,ab;q^2\right)_{\infty}}
{\left(aq^2,bq^2;q^2\right)_{\infty}}\sum_{n=0}^{\infty}
\frac{\left(-q,q^2/a,q^2/b;q^2\right)_{n}}
{\left(q;q\right)_{2n+1}\left(-q^2;q^2\right)_{n}}a^nb^n\nonumber\\
&\quad=\sum_{n=0}^{\infty}\left(1+q^{2n+1}\right)
\frac{\left(q^2/a,q^2/b;q^2\right)_{n}}
{\left(aq^2,bq^2;q^2\right)_n}
a^nb^nq^{n^2},\label{thm5}
\end{align}
where $\max\{|ab|, |aq^2|,|bq^2|\}<1$.

Note that
\begin{equation*}
\sum_{n=0}^{\infty}\frac{\left(-q;q\right)_{2n}}
{\left(q;q^2\right)_{n+1}\left(q^4;q^4\right)_{n}}q^{n\left(n+1\right)}
=\sum_{n=0}^{\infty}\frac{\left(-q;q^2\right)_{n}}
{\left(q;q\right)_{2n+1}}q^{n\left(n+1\right)}.
\end{equation*}
Hence, we take $\left(a,b\right)\to\left(0,-1\right)$ into \eqref{thm5} and use the Jacobi triple product identity \eqref{Jacobi} with $\left(q,z\right)\to\left(q^4,q\right)$ to obtain \eqref{S.11}.

Also, we observe that
\begin{equation*}
\sum_{n=0}^{\infty}\frac{\left(q^2;q^4\right)_nq^{2n\left(n+1\right)}}
{\left(q^4;q^4\right)_{n}\left(q;q^2\right)_{n}\left(q;q^2\right)_{n+1}}
=\sum_{n=0}^{\infty}\frac{\left(-q;q^2\right)_nq^{2n\left(n+1\right)}}
{\left(q;q\right)_{2n+1}\left(-q^2;q^2\right)_{n}}.
\end{equation*}
Then, we take $\left(a,b\right)\to\left(0,0\right)$ into \eqref{thm5} and use the Jacobi triple product identity \eqref{Jacobi} with $\left(q,z\right)\to\left(q^6,q\right)$ to obtain \eqref{S.87}.

Moreover, by inserting $\left(a,b\right)\to\left(0,-q\right)$ into \eqref{thm5}, we arrive at \eqref{new3}. 
\end{proof}

\begin{corollary}
We have
\begin{align}
\sum_{n=0}^{\infty}\frac{\left(-q;q\right)_nq^{n^2+n}}
{\left(q;q\right)_n\left(q;q^2\right)_{n+1}}
&=
\frac{\left(-q;q\right)_{\infty}\left(q,q^5,q^6;q^6\right)_{\infty}}
{\left(q;q\right)_{\infty}},\tag{S.22}\label{S.22}\\
\sum_{n=0}^{\infty}\frac{\left(-q^2;q^2\right)_n}
{\left(q;q\right)_{2n+1}}q^{n^2}
&=\frac{\left(-q;q^2\right)_{\infty}\left(-q,-q^3,q^4;q^4\right)_{\infty}}{\left(q;q^2\right)_{\infty}
}.\tag{\rm{\cite[Entry 1.7.13]{Ramanujan}}}\label{new6}
\end{align}
\end{corollary}
\begin{proof}
We take $\left(q,\alpha,c,d\right)\to\left(q^2,q^2,q,-q\right)$ into Theorem \ref{thm2} to arrive at
\begin{align}
\frac{\left(q^2,ab;q^2\right)_{\infty}}{\left(aq^2,bq^2;q^2\right)_{\infty}}
\sum_{n=0}^{\infty}\frac{\left(q^2/a,q^2/b,-q^2;q^2\right)_n}
{\left(q^2;q^2\right)_n\left(q^2;q^4\right)_{n+1}}a^nb^n
=
\sum_{n=0}^{\infty}\frac{\left(q^2/a,q^2/b;q^2\right)_n}
{\left(aq^2,bq^2;q^2\right)_n}a^nb^nq^{n^2+n},\label{thm3}
\end{align}
where $\max\left\{|ab|,|aq^2|,|bq^2|\right\}<1$.

Taking $\left(a,b\right)\to\left(0,0\right)$ into \eqref{thm3} and using the Jacobi triple product identity \eqref{Jacobi} with $\left(q,z\right)\to\left(q^6,1\right)$, we see that
\begin{align}
\left(q^2;q^2\right)_{\infty}
\sum_{n=0}^{\infty}\frac{\left(-q^2;q^2\right)_nq^{2n^2+2n}}
{\left(q^2;q^2\right)_n\left(q^2;q^4\right)_{n+1}}
=\left(-q^6,-q^6,q^6;q^6\right)_{\infty}.\label{e2}
\end{align}
After replacing $q^2$ by $q$, \eqref{e2} reduces to \eqref{S.22}.

Taking $\left(a,b\right)\to\left(0,-1/q\right)$ into \eqref{thm3} and using the Jacobi triple product identity \eqref{Jacobi} with $\left(q,z\right)\to\left(q^4,q\right)$, we arrive at the identity \ref{new6}.
\end{proof}
\begin{corollary}
We have
\begin{align}
\sum_{n=0}^{\infty}\frac{\left(q^2;q^2\right)_{n-1}q^{n^2}}
{\left(q;q\right)_n\left(q;q\right)_{n-1}\left(q;q^2\right)_n}
&=\frac{\left(-q^5,-q^7,q^{12};q^{12}\right)_{\infty}}{\left(q;q\right)_{\infty}}
,\tag{S.58}\label{S.58}\\
\sum_{n=0}^{\infty}\frac{\left(-q;q^2\right)_{n}\left(q^4;q^4\right)_{n-1}q^{n^2}}
{\left(q^2;q^2\right)_n\left(q^2;q^2\right)_{n-1}\left(q^2;q^4\right)_n}
&=\frac{\left(-q^6,-q^{10},q^{16};q^{16}\right)_{\infty}
\left(-q;q^2\right)_{\infty}}
{\left(q^2;q^2\right)_{\infty}}
.\tag{S.72}\label{S.72}
\end{align}
\end{corollary}
\begin{proof}
Taking $\left(\alpha,c,d\right)\to\left(1,q^{1/2},-q^{1/2}\right)$ into Theorem \ref{thm2} , 
 we have
\begin{align}
&\frac{\left(q,ab/q;q\right)_{\infty}}{\left(a,b;q\right)_{\infty}}
\sum_{n=0}^{\infty}\frac{\left(q/a,q/b,-1;q\right)_{n}}
{\left(q;q\right)_{n}\left(q;q^2\right)_{n}}
\left(ab/q\right)^n\nonumber\\
&\quad=1+\sum_{n=1}^{\infty}\left(1+q^{n}\right)
\frac{\left(q/a,q/b;q\right)_{n}}{\left(a,b;q\right)_{n}}
a^{n}b^{n}q^{\frac{n^2-3n}{2}},\label{C16}
\end{align}
where $\max\{|a|,|b|,|ab/q|\}<1$.

Taking $\left(a,b\right)\to\left(0,0\right)$ into \eqref{C16}, we have
\begin{align}
&\left(q;q\right)_{\infty}\sum_{n=0}^{\infty}
\frac{\left(-1;q\right)_nq^{n^2}}
{\left(q;q\right)_n\left(q;q^2\right)_n}\nonumber\\
&\quad=
\sum_{n=-\infty}^{\infty}q^{6n^2+n}+q\sum_{n=-\infty}^{\infty}q^{6n^2+5n}\nonumber\\
&\quad=
\left(-q^5,-q^7,q^{12};q^{12}\right)_{\infty}
+q\left(-q,-q^{11},q^{12};q^{12}\right)_{\infty}.
\label{equ5}
\end{align}
Consequently, we find that
\begin{align*}
&\sum_{n=0}^{\infty}\frac{\left(q^2;q^2\right)_{n-1}q^{n^2}}
{\left(q;q\right)_n\left(q;q\right)_{n-1}\left(q;q^2\right)_n}\nonumber\\
&\quad=\frac{1}{2}+\frac{1}{2}\sum_{n=0}^{\infty}
\frac{\left(-1;q\right)_nq^{n^2}}{\left(q;q\right)_n
\left(q;q^2\right)_n}\nonumber\\
&\quad=\frac{1}{2}+\frac{\left(-q^5,-q^7,q^{12};q^{12}\right)_{\infty}
+q\left(-q,-q^{11},q^{12};q^{12}\right)_{\infty}}{2\left(q;q\right)_{\infty}}\\
&\quad=\left(-q^5,-q^7,q^{12};q^{12}\right)_{\infty},
\end{align*}
where we have used
\begin{align*}
\left(q;q\right)_{\infty}
&=\sum_{n=-\infty}^{\infty}\left(-1\right)^nq^{\frac{3n^2-n}{2}}\nonumber\\
&=\sum_{n=-\infty}^{\infty}q^{6n^2-n}-q\sum_{n=-\infty}^{\infty}q^{6n^2+5n}\nonumber\\
&=\left(-q^5,-q^7,q^{12};q^{12}\right)_{\infty}
-q\left(-q,-q^{11},q^{12};q^{12}\right)_{\infty}.
\end{align*}
Therefore, the proof of \eqref{S.58} is completed.

Taking $\left(a,b\right)\to\left(0,-q^{1/2}\right)$ into \eqref{C16}, we have
\begin{align}
&\frac{\left(q;q\right)_{\infty}}{\left(-q^{1/2};q\right)_{\infty}}
\sum_{n=0}^{\infty}
\frac{\left(-1;q\right)_n}
{\left(q;q\right)_n\left(q^{1/2};q\right)_n}q^{{n^2}/{4}}\nonumber\\
&\quad=
\sum_{n=-\infty}^{\infty}q^{4n^2-n}
+q^{{1}/{2}}\sum_{n=-\infty}^{\infty}q^{4n^2+3n}\nonumber\\
&\quad
=\left(-q^3,-q^{5},q^{8};q^{8}\right)_{\infty}
+q^{{1}/{2}}\left(-q,-q^{7},q^{8};q^{8}\right)_{\infty}.
\label{equ6}
\end{align}
Therefore, we can prove \eqref{S.72} as follows,
\begin{align*}
&\sum_{n=0}^{\infty}\frac{\left(-q;q^2\right)_{n}\left(q^4;q^4\right)_{n-1}}
{\left(q^2;q^2\right)_n\left(q^2;q^2\right)_{n-1}\left(q^2;q^4\right)_n}q^{n^2}\\
&\quad=\frac{1}{2}+\frac{1}{2}
\sum_{n=0}^{\infty}\frac{\left(-1;q^2\right)_n}{\left(q,q^2;q^2\right)_n}q^{n^2}\\
&\quad=\frac{1}{2}+\frac{
\left(-q;q^2\right)_{\infty}
}
{2\left(q^2;q^2\right)_{\infty}}\left(\left(-q^6,-q^{10},q^{16};q^{16}\right)_{\infty}
+q\left(-q^2,-q^{14},q^{16};q^{16}\right)_{\infty}\right)\\
&\quad=\frac{\left(-q;q^2\right)_{\infty}\left(-q^6,-q^{10},q^{16};q^{16}\right)_{\infty}
}
{\left(q^2;q^2\right)_{\infty}},
\end{align*}
where we have used \eqref{equ6} by replacing $q$ by $q^2$ and the following result:
\begin{align*}
\frac{\left(q^2;q^2\right)_{\infty}}{\left(-q;q^2\right)_{\infty}}
&=\sum_{n=-\infty}^{\infty}\left(-1\right)^nq^{2n^2-n}\\
&=\sum_{n=-\infty}^{\infty}q^{8n^2-2n}-
\sum_{n=-\infty}^{\infty}q^{8n^2+6n}\\
&=\left(-q^6,-q^{10},q^{16};q^{16}\right)_{\infty}
-q\left(-q^2,-q^{14},q^{16};q^{16}\right)_{\infty}.
\end{align*}
\end{proof}

\section{Identities from a terminating $q$-analogue of Whipple's ${_3F_2}$ sum}
In this section,  we derive some identities from  an useful summation formula which can be found in \cite{Gasper_and_Rahman},
\begin{equation}\label{Whipple}
_4\phi_3\left(
\begin{array}{cccc}
q^{-n},\,q^{n+1},\,c,\,-c\\
e,\,c^2q/e,\,-q
\end{array};q,\,q\right)
=\frac{\left(eq^{-n},eq^{1+n},c^2q^{1-n}/e,c^2q^{n+2}/e;q^2\right)_{\infty}}
{\left(e,c^2q/e;q\right)_{\infty}}q^{n\left(n+1\right)/2}.
\end{equation}

\begin{theorem}\label{thm11}
For $\max\{|ab|,|aq^2|,|bq^2|\}<1$, we have
\begin{align*}
&\frac{\left(q^2,ab;q^2\right)_{\infty}}{\left(aq^2,bq^2;q^2\right)_{\infty}}
\sum_{n=0}^{\infty}
\frac{\left(q^2/a,q^2/b;q^2\right)_n}{\left(-q^2,q^2,-q;q^2\right)_n}a^nb^n\\
&\quad=\sum_{n=0}^{\infty}\left(-1\right)^n\left(1-q^{4n+2}\right)
\frac{\left(q^2/a,q^2/b;q^2\right)_n}{\left(aq^2,bq^2;q^2\right)_n}
a^nb^nq^{\frac{3n^2-n}{2}}.
\end{align*}
\end{theorem}
\begin{proof}
Making the substitution $\left(q,c,e\right)\to\left(q^2,0,-q\right)$ into \eqref{Whipple}, we have
\begin{equation}\label{A13}
{_3\phi_2}\left(
\begin{array}{ccc}
q^{-2n},q^{2n+2},0\\
-q^2,-q
\end{array};q^2,q^2\right)
=
\frac{\left(-q^{1-2n},-q^{2n+3};q^4\right)_{\infty}}{\left(-q;q^2\right)_{\infty}}
q^{n\left(n+1\right)}=q^{\frac{n^2+n}{2}}.
\end{equation}
Then, we substitute $\left(q,\alpha,\beta,c,d\right)\to\left(q^2,q^2,0,-q^2,-q\right)$ into \eqref{Formula1} to obtain the desired theorem.
\end{proof}
\begin{corollary}
We have
\begin{align}
\sum_{n=0}^{\infty}\frac{q^{n\left(n+2\right)}}{\left(q^4;q^4\right)_n}
&=\frac{\left(-q;q^2\right)_{\infty}\left(q,q^4,q^5;q^5\right)_{\infty}}
{\left(q^2;q^2\right)_{\infty}}
,\tag{S.16}\label{S.16}\\
\sum_{n=0}^{\infty}\frac{q^{2n\left(n+1\right)}}
{\left(q^2;q^2\right)_n\left(-q;q\right)_{2n}}
&=\frac{\left(q^2,q^5,q^7;q^7\right)_{\infty}}
{\left(q^2;q^2\right)_{\infty}},\tag{S.32}\label{S.32}\\
\sum_{n=0}^{\infty}\frac{\left(-1\right)^n\left(q;q^2\right)_nq^{n^2+2n}}
{\left(-q;q^2\right)_{n}\left(q^4;q^4\right)_n}
&=\frac{\left(q;q^2\right)_{\infty}\left(-q,-q^4,q^5;q^5\right)_{\infty}}
{\left(q^2;q^2\right)_{\infty}}.\tag{\rm{\cite[Eq. (2.23)]{Bowman}}}\label{e6}
\end{align}
\end{corollary}
\begin{proof}
By substituting $\left(a,b\right)\to\left(0,-q\right)$, $\left(a,b\right)\to\left(0,0\right)$ and $\left(a,b\right)\to\left(0,q\right)$ into Theorem \ref{thm11}, respectively, we obtain \eqref{S.16}, \eqref{S.32} and the identity \ref{e6}.
\end{proof}

\begin{theorem}\label{thm10}
For $\max\{|ab|,|aq^2|,|bq^2|\}<1$, we have
\begin{align*}
&\frac{\left(q^2,ab;q^2\right)_{\infty}}{\left(aq^2,bq^2;q^2\right)_{\infty}}
\sum_{n=0}^{\infty}
\frac{\left(q^2/a,q^2/b;q^2\right)_n}{\left(-q^2,q^2;q^2\right)_n\left(-q;q^2\right)_{n+1}}a^nb^n\\
&\quad=\sum_{n=0}^{\infty}\left(-1\right)^n\left(1-q^{2n+1}\right)
\frac{\left(q^2/a,q^2/b;q^2\right)_n}{\left(aq^2,bq^2;q^2\right)_n}
a^nb^nq^{\frac{3n^2+n}{2}}.
\end{align*}
\end{theorem}
\begin{proof}
Making the substitution $\left(q,c,e\right)\to\left(q^2,0,-q^3\right)$ into \eqref{Whipple}, we have
\begin{align}
{_3\phi_2}\left(
\begin{array}{ccc}
q^{-2n},q^{2n+2},0\\
-q^2,-q^3
\end{array};q^2,q^2\right)
=
\frac{1+q}{1+q^{2n+1}}q^{\frac{n^2+3n}{2}},\label{A18}
\end{align}
which can also be found in \cite[Eq. (3.13)]{Wang}.
Based on the above identity,  we substitute $\left(q,\alpha,\beta,c,d\right)\to\left(q^2,q^2,0,-q^2,-q^3\right)$ into \eqref{Formula1} to obtain the desired theorem.
\end{proof}
\begin{corollary}
We have
\begin{align}
\sum_{n=0}^{\infty}\frac{q^{n\left(n+1\right)}}
{\left(q^2;q^2\right)_n\left(-q;q^2\right)_{n+1}}
&=
\frac{\left(-q^2;q^2\right)_{\infty}\left(q,q^4,q^5;q^5\right)_{\infty}}
{\left(q^2;q^2\right)_{\infty}},\tag{S.17}\label{S.17}\\
\sum_{n=0}^{\infty}\frac{\left(-1\right)^n\left(q;q^2\right)_nq^{n^2+2n}}
{\left(-q;q^2\right)_{n+1}\left(q^4;q^4\right)_n}
&=\frac{\left(q;q^2\right)_{\infty}\left(-q^5,-q^5,q^5;q^5\right)_{\infty}}
{\left(q^2;q^2\right)_{\infty}},\tag{\rm{\cite[Eq. (2.5)]{McLaughlin}}}\label{new1}\\
\sum_{n=0}^{\infty}\frac{q^{n^2}}{\left(q^4;q^4\right)_n}
&=\frac{\left(-q;q^2\right)_{\infty}\left(q^2,q^3,q^5;q^5\right)_{\infty}}
{\left(q^2;q^2\right)_{\infty}},\tag{S.20}\label{S.20}\\
\sum_{n=0}^{\infty}\frac{q^{2n\left(n+1\right)}}
{\left(q^2;q^2\right)_n\left(-q;q\right)_{2n+1}}
&=\frac{\left(q,q^6,q^7;q^7\right)_{\infty}}
{\left(q^2;q^2\right)_{\infty}}.\tag{S.31}\label{S.31}
\end{align}
\end{corollary}
\begin{proof}
Substituting $\left(a,b\right)\to\left(0,-1\right)$ into Theorem \ref{thm10} and taking $\left(q,z\right)\to\left(q^5, -q\right)$ into the Jacobi triple product identity \eqref{Jacobi},, we deduce that \eqref{S.17} is true.

Moreover, setting $\left(a,b\right)\to\left(0,q\right)$ in Theorem \ref{thm10} and taking $\left(q,z\right)\to\left(q^5,1\right)$ into the Jacobi triple product identity \eqref{Jacobi}, we deduce that \ref{new1} is true.

Likewise, letting $\left(a,b\right)\to\left(0,-1/q\right)$ in Theorem \ref{thm10} and inserting $\left(q,z\right)\to\left(q^5,-q^2\right)$ into the Jacobi triple product identity \eqref{Jacobi}, we obtain \eqref{S.20}.

Similarly, substituting $\left(a,b\right)\to\left(0,0\right)$ into Theorem \ref{thm10} and using the Jacobi triple product identity \eqref{Jacobi} with $\left(q,z\right)\to\left(q^7,-q\right)$, we get \eqref{S.31}.
\end{proof}
\begin{theorem}\label{thm12}
For $\max\{|ab|,|aq^2|,|bq^2|\}<1$, we have
\begin{align*}
&\frac{\left(q^2,ab;q^2\right)_{\infty}}{\left(aq^2,bq^2;q^2\right)_{\infty}}
\sum_{n=0}^{\infty}
\frac{\left(q^2/a,q^2/b;q^2\right)_n}{\left(-q^2,q^2;q^2\right)_n\left(q;q^2\right)_{n+1}}a^nb^n\\
&\quad=\sum_{n=0}^{\infty}\left(-1\right)^n\left(1+q^{4n+1}\right)
\frac{\left(q^2/a,q^2/b;q^2\right)_{2n}}{\left(aq^2,bq^2;q^2\right)_{2n}}
a^{2n}b^{2n}q^{6n^2+n}\\
&\qquad-
\sum_{n=0}^{\infty}\left(-1\right)^n\left(1+q^{4n+3}\right)
\frac{\left(q^2/a,q^2/b;q^2\right)_{2n+1}}{\left(aq^2,bq^2;q^2\right)_{2n+1}}
a^{2n+1}b^{2n+1}q^{6n^2+7n+2}.
\end{align*}
\end{theorem}
\begin{proof}
Making the substitution $\left(q,c,e\right)\to\left(q^2,0,q^3\right)$ into \eqref{Whipple}, we have
\begin{align}
{_3\phi_2}\left(
\begin{array}{ccc}
q^{-2n},q^{2n+2},0\\
-q^2,q^3
\end{array};q^2,q^2\right)
&=
\frac{\left(q^{3-2n}, q^{5+2n};q^4\right)_{\infty}}
{\left(q^{3};q^2\right)_{\infty}}q^{n\left(n+1\right)}\nonumber\\
&
=\left\{\begin{array}{ll}
\left(-1\right)^mq^{2m^2+3m}\frac{1-q}{1-q^{4m+1}},&\quad {n=2m},\\
\left(-1\right)^{m}q^{2m^2+5m+2}\frac{1-q}{1-q^{4m+3}},&\quad{n=2m+1}.
\end{array}\right.\label{A19}
\end{align}
Based on the above identity,  we substitute $\left(q,\alpha,\beta,c,d\right)\to\left(q^2,q^2,0,-q^2,q^3\right)$ into \eqref{Formula1} to obtain the desired theorem.
\end{proof}
\begin{corollary}
We have
\begin{align}
&\sum_{n=0}^{\infty}\frac{\left(-q;q^2\right)_{n}}
{\left(q;q^2\right)_{n+1}\left(q^4;q^4\right)_n}q^{n\left(n+2\right)}
=\frac{\left(-q^2;q^2\right)_{\infty}
}
{\left(q^2;q^2\right)_{\infty}}\left(q^5,q^{15},q^{20};q^{20}\right)_{\infty},\tag{\cite[Eq (2.7)]{McLaughlin}}\label{N3}\\
&\sum_{n=0}^{\infty}\frac{q^{n\left(n+1\right)}}
{\left(q;q\right)_{2n+1}}
=\frac{\left(-q^2;q^2\right)_{\infty}
}
{\left(q^2;q^2\right)_{\infty}}\left(\left(q^7,q^{13},q^{20};q^{20}\right)_{\infty}
+q\left(q^3,q^{17},q^{20};q^{20}\right)_{\infty}\right).\tag{S.94}\label{S.94}
\end{align}
\end{corollary}
\begin{proof}
Taking $\left(a,b\right)\to\left(0,-q\right)$ and  $\left(a,b\right)\to\left(0,-1\right)$ in Theorem \ref{thm12}, respectively,  we arrive at \eqref{N3} and \eqref{S.94}.
\end{proof}
\begin{theorem}\label{thm13}
For $\max\{|ab|,|aq^2|,|bq^2|\}<1$, we have
\begin{align*}
&\frac{\left(q^2,ab;q^2\right)_{\infty}}{\left(aq^2,bq^2;q^2\right)_{\infty}}
\sum_{n=0}^{\infty}
\frac{\left(q^2/a,q^2/b;q^2\right)_n}{\left(q,q^2,-q^2;q^2\right)_n}a^nb^n\\
&\quad=\sum_{n=0}^{\infty}\left(-1\right)^n\left(1-q^{8n+2}\right)
\frac{\left(q^2/a,q^2/b;q^2\right)_{2n}}{\left(aq^2,bq^2;q^2\right)_{2n}}
a^{2n}b^{2n}q^{6n^2-n}\\
&\qquad+
\sum_{n=0}^{\infty}\left(-1\right)^n\left(1-q^{8n+6}\right)
\frac{\left(q^2/a,q^2/b;q^2\right)_{2n+1}}{\left(aq^2,bq^2;q^2\right)_{2n+1}}
a^{2n+1}b^{2n+1}q^{6n^2+5n+1}.
\end{align*}
\end{theorem}
\begin{proof}
Making the substitution $\left(q,c,e\right)\to\left(q^2,0,q\right)$ into \eqref{Whipple}, we have
\begin{align}
{_3\phi_2}\left(
\begin{array}{ccc}
q^{-2n},q^{2n+2},0\\
-q^2,q
\end{array};q^2,q^2\right)
&=
\frac{\left(q^{1-2n}, q^{3+2n};q^4\right)_{\infty}}
{\left(q;q^2\right)_{\infty}}q^{n\left(n+1\right)}\nonumber\\
&
=\left\{\begin{array}{ll}
\left(-1\right)^mq^{2m^2+m},&\quad {n=2m},\\
\left(-1\right)^{m+1}q^{2m^2+3m+1},&\quad{n=2m+1}.
\end{array}\right.\label{A20}
\end{align}
Based on the above identity,  we substitute $\left(q,\alpha,\beta,c,d\right)\to\left(q^2,q^2,0,-q^2,q\right)$ into \eqref{Formula1} to obtain the desired theorem.
\end{proof}
\begin{corollary}
\begin{align}
&\sum_{n=0}^{\infty}
\frac{\left(-1\right)^n\left(q;q^2\right)_n}{\left(-q;q^2\right)_n\left(q^4;q^4\right)_n}q^{n\left(n+2\right)}
=\frac{\left(q;q^2\right)_{\infty}\left(-q,-q^4,q^5;q^5\right)_{\infty}}{\left(q^2;q^2\right)_{\infty}},\tag{\rm{\cite[Eq. (2.23)]{Bowman}}}\label{N6}\\
&\sum_{n=0}^{\infty}\frac{q^{n\left(n+1\right)}}
{\left(q;q\right)_{2n}}
=\frac{\left(q^9,q^{11},q^{20};q^{20}\right)_{\infty}
-q^2\left(q,q^{19},q^{20};q^{20}\right)_{\infty}}
{\left(-q;q^2\right)_{\infty}\left(q;q\right)_{\infty}}.\tag{S.99}\label{S.99}
\end{align}
\end{corollary}
\begin{proof}
Inserting 
$\left(a,b\right)\to\left(0,-q\right)$ in Theorem \ref{thm13}, we get
\begin{align*}
&\frac{\left(q^2;q^2\right)_{\infty}}{\left(-q;q^2\right)_{\infty}}
\sum_{n=0}^{\infty}\frac{\left(-q;q^2\right)_n}{\left(q,q^2,-q^2;q^2\right)_n}q^{n^2+2n}\\
&\quad=\sum_{n=0}^{\infty}\left(-1\right)^n\left(1-q^{4n+1}\right)q^{10n^2+3n}
+\sum_{n=0}^{\infty}\left(-1\right)^n\left(1-q^{4n+3}\right)q^{10n^2+13n+4}\\
&\quad=\sum_{n=-\infty}^{\infty}\left(-1\right)^nq^{10n^2+3n}-q\sum_{n=0}^{\infty}q^{10n^2+7n}\\
&\quad=\left(q^7,q^{13},q^{20};q^{20}\right)_{\infty}-q\left(q^3,q^{17},q^{20};q^{20}\right)_{\infty}.
\end{align*}
When replacing $q$ by $-q$,   the above identity is rewritten as
\begin{align*}
&\frac{\left(q^2;q^2\right)_{\infty}}{\left(q;q^2\right)_{\infty}}
\sum_{n=0}^{\infty}\frac{\left(q;q^2\right)_n}{\left(-q,q^2,-q^2;q^2\right)_n}q^{n^2+2n}\\
&\quad=\sum_{n=-\infty}^{\infty}q^{10n^2+3n}+q^4\sum_{n=0}^{\infty}q^{10n^2+13n}\\
&\quad=\sum_{n=-\infty}^{\infty}q^{\frac{5n^2+3n}{2}}\\
&\quad=\left(-q,-q^4,q^5;q^5\right)_{\infty},
\end{align*}
which means \ref{N6} is true.

Similarly, setting $\left(a,b\right)\to\left(0,-1\right)$ in Theorem \ref{thm13}, we obtain \eqref{S.99}.
\end{proof}
When $\lambda=e$,
Theorem \ref{thm1} reduces to Liu's
another transformation formula with $z=1$ \cite{Liu_1}: for $|\alpha ab/q|<1$,
\begin{align}
&\frac{\left(\alpha q,\alpha ab/q;q\right)_{\infty}}{\left(\alpha a,\alpha b;q\right)_{\infty}}
{_4\phi_3}\left(\begin{array}{cccc}
q/a,\,q/b,\,\beta,\,\gamma\\
c,\,d,\,h
\end{array}
;q,{\alpha ab}/{q}\right)\nonumber\\
&\quad=
\sum_{n=0}^{\infty}
\frac{\left(1-\alpha q^{2n}\right)\left(\alpha,q/a,q/b;q\right)_n\left(-\alpha ab\right)^nq^{{n\left(n-3\right)}/{2}}}
{\left(1-\alpha\right)\left(q,\alpha a,\alpha b;q\right)_n}
{_4\phi_3}\left(\begin{array}{cccc}
q^{-n},\,\alpha q^n,\,\beta,\,\gamma\\
c,\,d,\,h
\end{array}
;q,q\right).\label{Formula2}
\end{align}
Next, in terms of \eqref{Formula2}, we establish  a generalized form of Theorem \ref{thm11} as follows.
\begin{theorem}\label{T5}
For $\max\{|c|, |ab|,|aq^2|, |bq^2|,  |\beta^2q^2/c|\}<1$, we have
\begin{align*}
&\frac{\left(q^2,ab;q^2\right)_{\infty}}{\left(aq^2,bq^2;q^2\right)_{\infty}}
\sum_{n=0}^{\infty}\frac{\left(q^2/a,q^2/b,\beta,-\beta;q^2\right)_n}
{\left(q^2,-q^2,c,\beta^2q^2/c;q^2\right)_n}a^nb^n\\
&\quad=
\sum_{n=0}^{\infty}
\left(-1\right)^n\left(1-q^{4n+2}\right)
\frac{\left(q^2/a,q^2/b;q^2\right)_n
\left(cq^{-2n},\beta^2q^{2-2n}/c;q^4\right)_n}
{\left(aq^2,bq^2;q^2\right)_n\left(c,\beta^2q^2/c;q^2\right)_n}a^nb^nq^{2n^2}.
\end{align*}
\end{theorem}
\begin{proof}
Setting $\left(q,c,e\right)\to\left(q^2,\beta,c\right)$ in \eqref{Whipple}, we get
\begin{align}
&_4\phi_3\left(\begin{array}{cccc}
q^{-2n},q^{2n+2},\beta,-\beta\\
c,\beta^2q^2/c,-q^2
\end{array};q^2\,q^2\right)\nonumber\\
&\quad=\frac{\left(cq^{-2n},cq^{2n+2},\beta^2q^{2-2n}/c,
\beta^2q^{2n+4}/c;q^4\right)_{\infty}}
{\left(c,\beta^2q^2/c;q^2\right)_{\infty}}q^{n\left(n+1\right)}\nonumber\\
&\quad=\frac{\left(cq^{-2n},\beta^2q^{2-2n}/c;q^4\right)_n}
{\left(c,\beta^2q^2/c;q^2\right)_{n}}q^{n^2+n}.\label{equ2}
\end{align}
Then, taking the substitution $q\to q^2, \alpha=q^2, \beta=-\gamma, cd=\beta^2q^2$ and $h=-q^2$  into \eqref{Formula2}, we obtain
\begin{align*}
&\frac{\left(q^2,ab;q^2\right)_{\infty}}{\left(aq^2,bq^2;q^2\right)_{\infty}}
\sum_{n=0}^{\infty}\frac{\left(q^2/a,q^2/b,\beta,-\beta;q^2\right)_n}
{\left(q^2,-q^2,c,\beta^2q^2/c;q^2\right)_n}a^nb^n\nonumber\\
&\quad=
\sum_{n=0}^{\infty}
\left(-1\right)^n\left(1-q^{4n+2}\right)
\frac{\left(q^2/a,q^2/b;q^2\right)_n}
{\left(aq^2,bq^2;q^2\right)_n}a^nb^nq^{n^2-n}
{_4\phi_3}\left(\begin{array}{cccc}
q^{-2n},q^{2n+2},\beta,-\beta\\
c,\beta^2q^2/c,-q^2
\end{array};q^2,\,q^2\right)\\
&\quad=
\sum_{n=0}^{\infty}
\left(-1\right)^n\left(1-q^{4n+2}\right)
\frac{\left(q^2/a,q^2/b;q^2\right)_n
\left(cq^{-2n},\beta^2q^{2-2n}/c;q^4\right)_n}
{\left(aq^2,bq^2;q^2\right)_n\left(c,\beta^2q^2/c;q^2\right)_n}a^nb^nq^{2n^2}
\end{align*}
as desired.
\end{proof}
\begin{corollary}
We have
\begin{align}
&\sum_{n=0}^{\infty}\frac{\left(-1;q^2\right)_n}{\left(q;q\right)_{2n}}q^{n^2+n}
\nonumber\\
&\quad
=\frac{\left(q^{5},q^{7},q^{12};q^{12}\right)_{\infty}
-q\left(q,q^{11},q^{12};q^{12}\right)_{\infty}}
{\left(q;q\right)_{\infty}},\tag{S.48}\label{S.48}\\
&\sum_{n=0}^{\infty}\frac{\left(-q;q^2\right)_n
\left(-1;q^4\right)_n}{\left(q^2;q^2\right)_{2n}}q^{n\left(n+2\right)}\nonumber\\
&\quad=\frac{\left(-q;q^2\right)_{\infty}}
{\left(q^2;q^2\right)_{\infty}}\left(\left(q^6,q^{10},q^{16};q^{16}\right)_{\infty}
-q\left(q^2,q^{14},q^{16};q^{16}\right)_{\infty}\right)
,\tag{S.67}\label{S.67}\\
&\sum_{n=0}^{\infty}\frac{\left(-q;q^2\right)_n
\left(-q^4;q^4\right)_{n-1}}
{\left(q^2;q^2\right)_{2n}}q^{n\left(n+2\right)}\nonumber\\
&\quad=\frac{\left(-q;q^2\right)_{\infty}
}
{\left(q^2;q^2\right)_{\infty}}\left(\left(-q^{28},-q^{36},q^{64};q^{64}\right)_{\infty}
-q\left(-q^{20},-q^{44},q^{64};q^{64}\right)_{\infty}\right)
.\tag{S.127}\label{S.127}
\end{align}
\end{corollary}
\begin{proof}
In order to prove \eqref{S.48},\eqref{S.67} and \eqref{S.127}, we first
insert $\left(c,\beta\right)=\left(q,\mathrm{i}\right)$ with $\mathrm{i}^2=-1$ in Theorem \ref{T5} to obtain
\begin{align}
&\frac{\left(q^2,ab;q^2\right)_{\infty}}{\left(aq^2,bq^2;q^2\right)_{\infty}}
\sum_{n=0}^{\infty}
\frac{\left(q^2/a,q^2/b;q^2\right)_n\left(-1;q^4\right)_n}{\left(q^2;q^2\right)_{2n}}a^nb^n\notag\\
&\quad=
\sum_{n=0}^{\infty}\left(-1\right)^n\left(1-q^{8n+2}\right)\frac{\left(q^2/a,q^2/b;q^2\right)_{2n}}
{\left(aq^2,bq^2;q^2\right)_{2n}}a^{2n}b^{2n}q^{4n^2-2n}\notag\\
&\qquad+
\sum_{n=0}^{\infty}\left(-1\right)^n\left(1-q^{8n+6}\right)\frac{\left(q^2/a,q^2/b;q^2\right)_{2n+1}}
{\left(aq^2,bq^2;q^2\right)_{2n+1}}a^{2n+1}b^{2n+1}q^{4n^2+2n}.\label{C4}
\end{align}
By taking $\left(a,b\right)\to\left(0,0\right)$  in \eqref{C4}, we deduce that
\begin{align*}
&\left(q^2;q^2\right)_{\infty}\sum_{n=0}^{\infty}
\frac{\left(-1;q^4\right)_n}{\left(q^2;q^2\right)_{2n}}q^{2n^2+2n}\\
&\quad=\sum_{n=-\infty}^{\infty}\left(-1\right)^nq^{12n^2+2n}
-q^2\sum_{n=-\infty}^{\infty}\left(-1\right)^nq^{12n^2+10n}\\
&\quad=\left(q^{10},q^{14},q^{24};q^{24}\right)_{\infty}
-q^2\left(q^{2},q^{22},q^{24};q^{24}\right)_{\infty}.
\end{align*}
Then, by $q^2\to q$, we can arrive at \eqref{S.48} .

Moreover, setting $\left(a,b\right)\to\left(0,-q\right)$ in \eqref{C4} leads to
\begin{align}
&\frac{\left(q^2;q^2\right)_{\infty}}{\left(-q;q^2\right)_{\infty}}
\sum_{n=0}^{\infty}\frac{\left(-q;q^2\right)_{n}
\left(-1;q^4\right)_{n}}
{\left(q^2;q^2\right)_{2n}}q^{n\left(n+2\right)}\nonumber\\
&\quad=\sum_{n=0}^{\infty}\left(-1\right)^n\left(1-q^{4n+1}\right)
q^{8n^2+2n}+\sum_{n=0}^{\infty}\left(-1\right)^n
\left(1-q^{4n+3}\right)q^{8n^2+10n+3}\nonumber\\
&\quad=\left(q^6,q^{10},q^{16};q^{16}\right)_{\infty}
-q\left(q^2,q^{14},q^{16};q^{16}\right)_{\infty},\nonumber
\end{align}
which means  \eqref{S.67} is true.

Similarly, we can prove \eqref{S.127} as follows,
\begin{align*}
&\sum_{n=0}^{\infty}\frac{\left(-q;q^2\right)_n
\left(-q^4;q^4\right)_{n-1}}
{\left(q^2;q^2\right)_{2n}}q^{n\left(n+2\right)}\nonumber\\
&\quad=1+\sum_{n=1}^{\infty}\frac{\left(-q;q^2\right)_n
\left(-q^4;q^4\right)_{n-1}}
{\left(q^2;q^2\right)_{2n}}q^{n\left(n+2\right)}\nonumber\\
&\quad=\frac{1}{2}+\frac{1}{2}\sum_{n=1}^{\infty}\frac{\left(-q;q^2\right)_n
\left(-1;q^4\right)_{n}}
{\left(q^2;q^2\right)_{2n}}q^{n\left(n+2\right)}\nonumber\\
&\quad=\frac{1}{2}+\frac{\left(-q;q^2\right)_{\infty}
}
{2\left(q^2;q^2\right)_{\infty}}\left(\left(q^6,q^{10},q^{16};q^{16}\right)_{\infty}
-q\left(q^2,q^{14},q^{16};q^{16}\right)_{\infty}\right),
\end{align*}
where we have used
\begin{align*}
&\left(q^6,q^{10},q^{16};q^{16}\right)_{\infty}
=\left(-q^{28},-q^{36},q^{64};q^{64}\right)_{\infty}
-q^6\left(-q^{4},-q^{60},q^{64};q^{64}\right)_{\infty},
\\
&\left(q^2,q^{14},q^{16};q^{16}\right)_{\infty}
=\left(-q^{20},-q^{44},q^{64};q^{64}\right)_{\infty}
+q^2\left(-q^{12},-q^{52},q^{64};q^{64}\right)_{\infty}
\end{align*}
and
\begin{align*}
\frac{\left(q^2;q^2\right)_{\infty}}{\left(-q;q^2\right)_{\infty}}
&=\sum_{n=-\infty}^{\infty}\left(-1\right)^nq^{2n^2-n}\\
&=\sum_{n=-\infty}^{\infty}\left(-1\right)^nq^{32n^2-4n}
-q\sum_{n=-\infty}^{\infty}\left(-1\right)^nq^{32n^2-12n}\\
&\quad-q^3\sum_{n=-\infty}^{\infty}\left(-1\right)^nq^{32n^2+20n}
+q^6\sum_{n=-\infty}^{\infty}\left(-1\right)^nq^{32n^2+28n}\\
&=\left(-q^{28},-q^{36},q^{64};q^{64}\right)_{\infty}
-q\left(-q^{20},-q^{44},q^{64};q^{64}\right)_{\infty}\\
&\quad-q^3\left(-q^{12},-q^{52},q^{64};q^{64}\right)_{\infty}
+q^6\left(-q^{4},-q^{60},q^{64};q^{64}\right)_{\infty}.
\end{align*}
\end{proof}
\begin{corollary}
We have
\begin{align}
&\sum_{n=0}^{\infty}\frac{\left(-q^2;q^2\right)_n
\left(-q^2;q^4\right)_n}
{\left(-1;q^2\right)_{n}\left(q;q^2\right)_{n+1}
\left(q^4;q^4\right)_{n}}q^{n\left(n+1\right)}\nonumber\\
&\quad=\frac{\left(-q^2;q^2\right)_{\infty}
}
{\left(q^2;q^2\right)_{\infty}}\left(\left(q^6,q^{10},q^{16};q^{16}\right)_{\infty}
+q\left(q^2,q^{14},q^{16};q^{16}\right)_{\infty}\right),\tag{S.65}\label{S.65}\\
&\sum_{n=0}^{\infty}\frac{\left(-q;q^2\right)_{n+1}
\left(-q^2;q^4\right)_n}
{\left(q^2;q^2\right)_{2n+1}}q^{n\left(n+2\right)}\nonumber\\
&\quad=\frac{\left(-q;q^2\right)_{\infty}
\left(q^4,q^{12},q^{16};q^{16}\right)_{\infty}}{\left(q^2;q^2\right)_{\infty}
}.\tag{S.70}\label{S.70}
\end{align}
\end{corollary}
\begin{proof}
Letting  $\left(c,\beta\right)=\left(-q,q\mathrm{i}\right)$ in Theorem \ref{T5}, we get
\begin{align}
&\frac{\left(q^2,ab;q^2\right)_{\infty}}{\left(aq^2,bq^2;q^2\right)_{\infty}}
\sum_{n=0}^{\infty}
\frac{\left(q^2/a,q^2/b;q^2\right)_n\left(-q^2;q^4\right)_n}{\left(-q,-q^2,q^2;q^2\right)_{n}\left(q;q^2\right)_{n+1}}a^nb^n\notag\\
&\quad=
\sum_{n=0}^{\infty}\left(-1\right)^n\left(1+q^{4n+1}\right)\frac{\left(q^2/a,q^2/b;q^2\right)_{2n}}
{\left(aq^2,bq^2;q^2\right)_{2n}}a^{2n}b^{2n}q^{4n^2}\notag\\
&\qquad-
\sum_{n=0}^{\infty}\left(-1\right)^n\left(1+q^{4n+3}\right)\frac{\left(q^2/a,q^2/b;q^2\right)_{2n+1}}
{\left(aq^2,bq^2;q^2\right)_{2n+1}}a^{2n+1}b^{2n+1}q^{4n^2+4n+1},\label{C5}
\end{align}
where $|ab|<1$.

By setting $\left(a,b\right)\to\left(0,-1\right)$ in \eqref{C5}, we can  prove \eqref{S.65} as follows,
\begin{align*}
&\frac{\left(q^2;q^2\right)_{\infty}}{\left(-q^2;q^2\right)_{\infty}}\sum_{n=0}^{\infty}\frac{\left(-q^2;q^2\right)_n
\left(-q^2;q^4\right)_nq^{n\left(n+1\right)}}
{\left(-1;q^2\right)_{n}\left(q;q^2\right)_{n+1}
\left(q^4;q^4\right)_{n}}\nonumber\\
&\quad=\sum_{n=0}^{\infty}\left(-1\right)^n\left(1+q^{4n+1}\right)
q^{8n^2+2n}-\sum_{n=0}^{\infty}\left(-1\right)^n
\left(1+q^{4n+3}\right)q^{8n^2+10n+3}\nonumber\\
&\quad={\left(q^6,q^{10},q^{16};q^{16}\right)_{\infty}
+q\left(q^2,q^{14},q^{16};q^{16}\right)_{\infty}}.
\end{align*}

Similarly, setting $\left(a,b\right)\to\left(0,-q\right)$ in \eqref{C5},  we can prove  \eqref{S.70} as follows,
\begin{align*}
&\sum_{n=0}^{\infty}\frac{\left(-q;q^2\right)_{n+1}
\left(-q^2;q^4\right)_n}
{\left(q^2;q^2\right)_{2n+1}}q^{n\left(n+2\right)}\\
&\quad=\frac{\left(-q;q^2\right)_{\infty}}{\left(q^2;q^2\right)_{\infty}}\left(\sum_{n=0}^{\infty}\left(-1\right)^n
q^{8n^2+4n}-\sum_{n=0}^{\infty}\left(-1\right)^n
q^{8n^2+12n+4}\right)\nonumber\\
&\quad=\frac{\left(-q;q^2\right)_{\infty}
\left(q^4,q^{12},q^{16};q^{16}\right)_{\infty}}{\left(q^2;q^2\right)_{\infty}}.
\end{align*}

\end{proof}

\section{Identities from a $_4\phi_3$ summation of Verma and Jain}
%
\begin{theorem}\label{T3}
For $\max\{|a|, |b|, |ab/q^2|\}<1$, we have
\begin{align*}
&\frac{\left(q^2,ab/q^2;q^2\right)_{\infty}}{\left(a,b;q^2\right)_{\infty}}
\sum_{n=0}^{\infty}\frac{\left(q^2/a,q^2/b;q^2\right)_n}
{\left(q^2,-q,-q^2;q^2\right)_n}\left(ab/q^2\right)^n\\
&\quad=1+\sum_{n=0}^{\infty}\left(-1\right)^n\left(1+q^n\right)
\frac{\left(q^2/a,q^2/b;q^2\right)_n}
{\left(a,b;q^2\right)_n}a^nb^nq^{\frac{3n^2-5n}{2}}.
\end{align*}
\end{theorem}
\begin{proof}
Recall an identity due to Verma and Jain \cite[Eq. (5.4)]{Verma}
\begin{equation}\label{Verma1}
_4\phi_3\left(\begin{array}{cccc}
q^{-2n},a^2q^{2n},-c, -cq\\
-aq,-aq^2,c^2
\end{array};q^2\,q^2\right)
=\frac{\left(-q,qa/c;q\right)_n\left(1+a\right)\left(-c\right)^n}
{\left(-a,c;q\right)_n\left(1+aq^{2n}\right)}.
\end{equation}
Letting $\left(a,c\right)\to\left(1,0\right)$ into \eqref{Verma1}, we have
\begin{equation*}
_3\phi_2\left(\begin{array}{ccc}
q^{-2n},q^{2n},0\\
-q,-q^2
\end{array};q^2\,q^2\right)=\frac{1+q^n}{1+q^{2n}}
q^{\frac{n\left(n+1\right)}{2}}.
\end{equation*}
Based on the above summation formula, we substitute $\left(q,\alpha,\beta,c,d\right)\to\left(q^2,1,0,-q,-q^2\right)$ into \eqref{Formula1} to obtain the desired result.
\end{proof}
\begin{corollary}
We have
\begin{align}
&\sum_{n=0}^{\infty}
\frac{\left(q;q^2\right)_n\left(-1\right)^nq^{n^2}}
{\left(-q;q^2\right)_n\left(q^4;q^4\right)_n}
=\frac{\left(-q^2,-q^3,q^5;q^5\right)_{\infty}
\left(q;q^2\right)_{\infty}}{\left(q^2;q^2\right)_{\infty}}, \tag{S.21}\label{S.21}\\
&\sum_{n=0}^{\infty}
\frac{q^{2n^2}}
{\left(q^2;q^2\right)_n\left(-q;q\right)_{2n}}
=\frac{\left(q^3,q^4,q^7;q^7\right)_{\infty}}
{\left(q^2;q^2\right)_{\infty}}.\tag{S.33}\label{S.33}
\end{align}
\end{corollary}
\begin{proof}
Taking $\left(a,b\right)\to\left(0,q\right)$ into Theorem \ref{T3} and using Jacobi triple product identity \eqref{Jacobi} with $\left(q,z\right)\to\left(q^5,q^2\right)$, we  obtain \eqref{S.21}.

Similarly, inserting $\left(a,b\right)\to\left(0,0\right)$ into Theorem \ref{T3} and applying Jacobi triple product identity \eqref{Jacobi} with $\left(q,z\right)\to\left(q^7,-q^3\right)$, we get \eqref{S.33}.
\end{proof}
Next, we generalize  Theorem \ref{T3} into the following form  which was first provided by  Liu in \cite[Proposition 11.4]{Liu_4}.
\begin{theorem}\label{T6}
For $\max\{|\beta|,|\alpha a|, |\alpha b|,  |\alpha^{1/2}q^2|, |\alpha ab/q^4|\}<1$, we have
\begin{align*}
&\frac{\left(\alpha q^4,\alpha ab/q^4;q^4\right)_{\infty}}{\left(\alpha a,\alpha b;q^4\right)_{\infty}}
\sum_{n=0}^{\infty}
\frac{\left(q^4/a,q^4/b,\beta,\beta q^2;q^4\right)_n}
{\left(q^4,-\alpha^{1/2}q^2,-\alpha^{1/2}q^4,\beta^2;q^4\right)_n}\left(\alpha ab/q^4\right)^n\\
&\quad=
\sum_{n=0}^{\infty}
\left(-1\right)^n
\frac{1-\alpha^{1/2}q^{4n}}{1-\alpha^{1/2}}
\frac{\left(\alpha^{1/2},-\alpha^{1/2}q^2/\beta;q^2\right)_n
\left(q^4/a,q^4/b;q^4\right)_n}
{\left(-\beta,q^2;q^2\right)_n\left(\alpha a,\alpha b;q^4\right)_n}\left(\alpha\beta ab\right)^nq^{2n^2-6n}.
\end{align*}
\end{theorem}
\begin{proof}
Replacing $q$ by $q^4$ and substituting $\left(c,d,h,z,\gamma\right)\to
\left(-\alpha^{1/2}q^2,\alpha^{1/2}q^4,\beta^2,1,\beta q^2\right)$ in \eqref{Formula2}, we deduce that
\begin{align*}
&\frac{\left(\alpha q^4,\alpha ab/q^4;q^4\right)_{\infty}}{\left(\alpha a,\alpha b;q^4\right)_{\infty}}
\sum_{n=0}^{\infty}
\frac{\left(q^4/a,q^4/b,\beta,\beta q^2;q^4\right)_n}
{\left(q^4,-\alpha^{1/2}q^2,-\alpha^{1/2}q^4,\beta^2;q^4\right)_n}\left(\alpha ab/q^4\right)^n\\
&\quad=
\sum_{n=0}^{\infty}\left(-1\right)^n\frac{1-\alpha q^{8n}}{1-\alpha}
\frac{\left(\alpha,q^4/a,q^4/b;q^4\right)_n}
{\left(q^4,\alpha a,\alpha b;q^4\right)_n}\left(\alpha ab\right)^nq^{2n^2-6n}\\
&\qquad \times
{_4\phi_3}\left(\begin{array}{cccc}
q^{-4n},\alpha q^{4n},\beta,\beta q^2\\
-\alpha^{1/2}q^2,-\alpha^{1/2}q^4,\beta^2
\end{array};q^4,\,q^4\right)\\
&\quad=\sum_{n=0}^{\infty}
\left(-1\right)^n
\frac{1-\alpha^{1/2}q^{4n}}{1-\alpha^{1/2}}
\frac{\left(\alpha^{1/2},-\alpha^{1/2}q^2/\beta;q^2\right)_n
\left(q^4/a,q^4/b;q^4\right)_n}
{\left(-\beta,q^2;q^2\right)_n\left(\alpha a,\alpha b;q^4\right)_n}\left(\alpha\beta ab\right)^nq^{2n^2-6n},
\end{align*}
where we have used \eqref{Verma1} with the substitution $\left(q,a,c\right)\to\left(q^2,\alpha^{1/2},-\beta\right)$.
\end{proof}
\begin{corollary}
We have
\begin{align}
\sum_{n=0}^{\infty}\frac{\left(q;q^2\right)_{2n}}
{\left(q^4;q^4\right)_{2n}}q^{4n^2}
&=\frac{\left(q^5,q^7,q^{12};q^{12}\right)_{\infty}}
{\left(q^4;q^4\right)_{\infty}},\tag{S.53}
\label{S.53}\\
\sum_{n=0}^{\infty}\frac{\left(-1\right)^n\left(q;q^2\right)_{2n}}
{\left(-q^2;q^4\right)_{n}\left(q^8;q^8\right)_{n}}q^{2n^2}
&=\frac{\left(q^2;q^4\right)_{\infty}\left(-q^3,-q^5,q^{8};q^{8}\right)_{\infty}}
{\left(q^4;q^4\right)_{\infty}},\tag{\cite[Eq. (2.7)]{Bowman}}\label{N7}\\
\sum_{n=0}^{\infty}\frac{\left(-q;q^2\right)_{2n}}
{\left(q^2;q^4\right)_{n}\left(q^8;q^8\right)_{n}}q^{2n^2}
&=\frac{\left(-q^2;q^4\right)_{\infty}\left(-q^3,-q^5,q^{8};q^{8}\right)_{\infty}}
{\left(q^4;q^4\right)_{\infty}}.\tag{\rm{\cite[Eq. (7.24)]{Gessel}}}\label{N8}
\end{align}
\end{corollary}

\begin{proof}
Taking $\left(\alpha,\beta\right)\to\left(1,q\right)$ into Theorem \ref{T6}, we see that
\begin{align}
&\frac{\left(q^4,ab/q^4;q^4\right)_{\infty}}{\left(a,b;q^4\right)_{\infty}}
\sum_{n=0}^{\infty}
\frac{\left(q,q^3,q^4/a,q^4/b;q^4\right)_n}{\left(q^4,-q^4,q^2,-q^2;q^4\right)_n}\left(ab/q^4\right)^n\nonumber\\
&\quad=1+\sum_{n=1}^{\infty}\left(-1\right)^n\left(1+q^{2n}\right)
\frac{\left(q^4/a,q^4/b;q^4\right)_n}{\left(a,b;q^4\right)_n}
a^nb^nq^{2n^2-5n}.\label{C8}
\end{align}
Then, taking $\left(a,b\right)\to\left(0,0\right)$ and $\left(a,b\right)\to\left(0,q^2\right)$ in \eqref{C8}, we arrive at  \eqref{S.53} and \ref{N7}.

In addition, taking $\left(a,b\right)\to\left(0,-q^2\right)$ in \eqref{C8}, we get
\begin{align*}
\sum_{n=0}^{\infty}\frac{\left(q;q^2\right)_{2n}}
{\left(q^2;q^4\right)_{n}\left(q^8;q^8\right)_{n}}q^{2n^2}
&=\frac{\left(-q^2;q^4\right)_{\infty}\left(q^3,q^5,q^{8};q^{8}\right)_{\infty}}
{\left(q^4;q^4\right)_{\infty}}.
\end{align*}
Replacing $q$ by $-q$ in the above identity, we immediately obtain \ref{N8}.
\end{proof}

\begin{corollary}
We have
\begin{align}
\sum_{n=0}^{\infty}\frac{\left(q;q^2\right)_{2n+1}}
{\left(q^4;q^4\right)_{2n+1}}q^{4n^2+4n}
&=\frac{\left(q,q^{11},q^{12};q^{12}\right)_{\infty}}{\left(q^4;q^4\right)_{\infty}},
\tag{S.55}\label{S.55}\\
\sum_{n=0}^{\infty}\frac{\left(-q;q^2\right)_{2n+1}}
{\left(q^4;q^4\right)_{2n+1}}q^{4n^2+4n}
&=\frac{\left(-q,-q^{11},q^{12};q^{12}\right)_{\infty}}{\left(q^4;q^4\right)_{\infty}},
\tag{S.57}\label{S.57}
\\
\sum_{n=0}^{\infty}\frac{\left(q;q^2\right)_{2n+1}\left(-q^4;q^4\right)_n}
{\left(q^4;q^4\right)_{2n+1}}q^{2n^2+2n}
&=\frac{\left(-q^4;q^4\right)_{\infty}\left(q,q^7,q^{8};q^{8}\right)_{\infty}}
{\left(q^4;q^4\right)_{\infty}}.\tag{\rm{\cite[Eq. (7.25)]{Gessel}}}\label{N9}
\end{align}
\end{corollary}

\begin{proof}
Taking $\left(\alpha,\beta\right)\to\left(q^4,q^3\right)$ into Theorem \ref{T6}, we obtain
\begin{align}
&\frac{\left(q^4,ab;q^4\right)_{\infty}}{\left(aq^4,bq^4;q^4\right)_{\infty}}
\sum_{n=0}^{\infty}\frac{\left(q;q^2\right)_{2n+1}\left(q^4/a,q^4/b;q^4\right)_n}{\left(q^4;q^4\right)_{2n+1}}a^nb^n\nonumber\\
&\quad=
\sum_{n=0}^{\infty}\left(-1\right)^n\left(1-q^{2n+1}\right)
\frac{\left(q^4/a,q^4/b;q^4\right)_n}{\left(aq^4,bq^4;q^4\right)_n}a^nb^nq^{2n^2+n}.\label{C9}
\end{align}

Setting $\left(a,b\right)\to\left(0,0\right)$ in \eqref{C9},
 we obtain \eqref{S.55}. When replacing $q$ by $-q$, \eqref{S.55} becomes to \eqref{S.57}.
 
Besides, inserting  $\left(a,b\right)\to\left(0,-1\right)$ in \eqref{C9},  we obtain \ref{N9} .
\end{proof}

\section{Identities from a ${_4\phi_3}$ summation of Andrews}
In  \cite{WangChun}, Wang and Chern established  the following $q$-series transformation formula.
\begin{theorem}\label{T4}
For $\max\{|\alpha a|,|\alpha b|,|\beta^2|,|\sqrt{\alpha q}|,|\alpha ab/q|\}<1$, we have
\begin{align*}
&\frac{\left(\alpha q,\alpha ab/q;q\right)_{\infty}}
{\left(\alpha a,\alpha b;q\right)_{\infty}}\sum_{n=0}^{\infty}
\frac{\left(q/a,q/b,\beta,-\beta;q\right)_{n}}
{\left(q,\sqrt{\alpha q},-\sqrt{\alpha q},\beta^2;q\right)_n}\left(\alpha ab/q\right)^n\nonumber\\
&\quad=\sum_{n=0}^{\infty}
\frac{\left(1-\alpha q^{4n}\right)\left(\alpha,q/a,q/b;q\right)_{2n}\left(q,\alpha q/\beta^2;q^2\right)_n}{\left(1-\alpha\right)
\left(q,\alpha a,\alpha b;q\right)_{2n}\left(\alpha q,\beta^2q;q^2\right)_n}\left(\alpha\beta ab\right)^{2n}q^{2n^2-3n}.
\end{align*}
\end{theorem}
To ensure the completeness of the article's content, we provide a proof of the above result. 
\begin{proof}[Proof of Theorem \ref{T4}]
Recall a summation formula due to Andrew \cite{Andrews1}:
\begin{align}
_4\phi_3\left(\begin{array}{cccc}
q^{-n},a^2q^{n+1},c,-c\\
aq,-aq,c^2
\end{array};q,\,q\right)=\left\{\begin{array}{ll}
0,&\quad n\equiv1\pmod2,\\
\frac{c^n\left(q,a^2q^2/c^2;q^2\right)_{n/2}}
{\left(a^2q^2,c^2q;q^2\right)_{n/2}},&\quad n\equiv0\pmod2.
\end{array}\right.\label{Andrews}
\end{align}
Taking $a=\sqrt{\alpha/q}, c=\beta$ into \eqref{Andrews}, we obtain
\begin{align*}
_4\phi_3\left(\begin{array}{cccc}
q^{-n},\alpha q^n,\beta,-\beta\\
\sqrt{\alpha q},-\sqrt{\alpha q},\beta^2
\end{array};q,\,q\right)
=\left\{\begin{array}{ll}
0,&\quad n\equiv1\pmod2,\\
\frac{\beta^n\left(q,\alpha q/\beta^2;q^2\right)_{n/2}}
{\left(\alpha q,\beta^2q;q^2\right)_{n/2}},& \quad{n\equiv0\pmod2}.
\end{array}\right.
\end{align*}
Based on the above summation formula, we take $c=-d=\sqrt{\alpha q}, \beta=-\gamma, h=\beta^2$ and $z=1$ into \eqref{Formula2} to get the desired result.
\end{proof}

\begin{corollary}
We have
\begin{align}
\sum_{n=0}^{\infty}\frac{q^{n^2}}{\left(q;q\right)_n
\left(q;q^2\right)_{n}}
&=\frac{\left(q^6,q^8,q^{14};q^{14}\right)_{\infty}}
{\left(q;q\right)_{\infty}},\tag{S.61}\label{S.61}\\
\sum_{n=0}^{\infty}\frac{q^{n^2}}{\left(q;q\right)_{2n}}
&=\frac{\left(q^2,q^8,q^{10};q^{10}\right)_{\infty}
\left(q^6,q^{14};q^{20}\right)_{\infty}}
{\left(q;q\right)_{\infty}}.\tag{S.98}\label{S.98}
\end{align}
\end{corollary}
\begin{proof}
Taking $\left(a,c\right)\to\left(q^{-1/2},0\right)$ into Andrews's formula \eqref{Andrews}, we have
\begin{equation*}
_3\phi_2\left(
\begin{array}{ccc}
q^{-n},q^{n},0\\
q^{1/2},-q^{1/2}
\end{array};q,\,q\right)
=\left\{\begin{array}{ll}
0,&\quad n\equiv1\pmod2,\\
\left(-1\right)^{{n}/{2}}
q^{{n^2}/{4}},&\quad n\equiv0\pmod2.
\end{array}
\right.
\end{equation*}
Based on the above summation, we fix $\alpha=1,\beta=0$ in  Theorem \ref{T4} to obtain
\begin{align}
&\frac{\left(q,ab/q;q\right)_{\infty}}{\left(a,b;q\right)_{\infty}}
\sum_{n=0}^{\infty}\frac{\left(q/a,q/b;q\right)_{n}}
{\left(q;q\right)_{n}\left(q;q^2\right)_{n}}
\left(ab/q\right)^n\nonumber\\
&\quad=1+\sum_{n=1}^{\infty}\left(-1\right)^n\left(1+q^{2n}\right)
\frac{\left(q/a,q/b;q\right)_{2n}}{\left(a,b;q\right)_{2n}}
a^{2n}b^{2n}q^{3n^2-3n}.\label{C17}
\end{align}
where $\max\{|a|,|b|,|ab/q|\}<1$.

Substituting $\left(a,b\right)\to\left(0,0\right)$ into \eqref{C17} and taking $\left(q,z\right)\to\left(q^{14},-q^6\right)$ in Jacobi triple product identity \eqref{Jacobi}, we obtain \eqref{S.61}.

Similarly, taking $\left(a,b\right)\to\left(0,-q^{1/2}\right)$ into \eqref{C17} and taking $\left(q,z\right)\to\left(q^{10},q^4\right)$ in Jacobi triple product identity \eqref{Jacobi}, we obtain
\begin{align}
\frac{\left(q;q\right)_{\infty}}{\left(-q^{1/2};q\right)_{\infty}}
\sum_{n=0}^{\infty}\frac{\left(-q^{1/2};q\right)_n}
{\left(q;q\right)_n\left(q;q^2\right)_{n}}q^{{n^2}/{2}}
=\left(q^4,q^6,q^{10};q^{10}\right)_{\infty}.\label{e5}
\end{align}
We further replace $q$ by $q^2$ in \eqref{e5}, which, after simplification, results in \eqref{S.98}.
\end{proof}

\begin{corollary}
We have
\begin{align}
\sum_{n=0}^{\infty}\frac{
\left(-q;q\right)_{n}\left(-q;q^2\right)_{n}
}{\left(q;q\right)_{2n+1}}q^{\frac{n\left(n+3\right)}{2}}
&=\frac{\left(-q;q\right)_{\infty}\left(q,q^7,q^8;q^8\right)_{\infty}}
{\left(q;q\right)_{\infty}}.\tag{S.35\,and\,S.106} \label{S.35}
\end{align}
\end{corollary}
\begin{proof}
When $\alpha=q^2,\beta=q^{1/2}{\rm i}$, Theorem \ref{T4} reduces to 
\begin{align}
&\frac{\left(q,abq;q\right)_{\infty}}{\left(aq^2,bq^2;q\right)_{\infty}}
\sum_{n=0}^{\infty}\frac{\left(q/a,q/b\right)_n\left(-q;q^2\right)_n}{\left(q;q\right)_{2n+1}}a^nb^nq^n
\nonumber\\
&\quad=
\sum_{n=0}^{\infty}\left(-1\right)^n\left(1-q^{4n+2}\right)\frac{\left(q/a,q/b;q\right)_{2n}}
{\left(aq^2,bq^2;q\right)_{2n}}\left(ab\right)^{2n}q^{2n^2+2n}.\label{C11}
\end{align}
By taking $\left(a,b\right)\to\left(0,-1\right)$ into \eqref{C11}, we obtain \eqref{S.35}.
\end{proof}
\begin{corollary}
We have
\begin{align}
\sum_{n=0}^{\infty}\frac{\left(-q;q\right)_n}{\left(q;q\right)_n
\left(q;q^2\right)_{n+1}}q^{\frac{n\left(n+3\right)}{2}}
&=\frac{\left(-q;q\right)_{\infty}\left(q,q^9,q^{10};q^{10}\right)_{\infty}}
{\left(q;q\right)_{\infty}},\tag{S.43}\label{S.43}
\\
\sum_{n=0}^{\infty}\frac{q^{n\left(n+2\right)}}
{\left(q;q\right)_n\left(q;q^2\right)_{n+1}}
&=\frac{\left(q^2,q^{12},q^{14};q^{14}\right)_{\infty}}{\left(q;q\right)_{\infty}}
,\tag{S.59}\label{S.59}
\end{align}
\end{corollary}
\begin{proof}
When $\alpha=q^2,\beta=0$, Theorem \ref{T4} reduces to 
\begin{align}
&\frac{\left(q,abq;q\right)_{\infty}}{\left(aq^2,bq^2;q\right)_{\infty}}
\sum_{n=0}^{\infty}\frac{\left(q/a,q/b\right)_n}{\left(q;q\right)_{n}\left(q;q^2\right)_{n+1}}a^nb^nq^n
\nonumber\\
&\quad=
\sum_{n=0}^{\infty}\left(-1\right)^n\left(1-q^{4n+2}\right)\frac{\left(q/a,q/b;q\right)_{2n}}
{\left(aq^2,bq^2;q\right)_{2n}}\left(ab\right)^{2n}q^{3n^2+3n}.\label{C12}
\end{align}
By taking $\left(a,b\right)\to\left(0,-1\right)$ and $\left(a,b\right)\to\left(0,0\right)$ into \eqref{C12}, respectively,  we obtain \eqref{S.43} and  \eqref{S.59}.
\end{proof}
\begin{corollary}
We have
\begin{align}
&\sum_{n=0}^{\infty}\frac{\left(-q^2;q^2\right)_{n-1}\left(1+q^n\right)}
{\left(q;q\right)_{2n}}q^{n^2}
=\frac{\left(q^5,q^7,q^{12};q^{12}\right)_{\infty}}
{\left(q;q\right)_{\infty}},\tag{S.54} \label{S.54}\\
&\sum_{n=0}^{\infty}\frac{\left(-q^4;q^4\right)_{n-1}\left(-q;q^2\right)_nq^{n^2}}
{\left(q^2;q^4\right)_n\left(q^2;q^2\right)_n\left(-q^2;q^2\right)_{n-1}}
=
\frac{\left(-q;q^2\right)_{\infty}\left(q^6,q^{10},q^{16};q^{16}\right)_{\infty}}
{\left(q^2;q^2\right)_{\infty}},\tag{S.71}\label{S.71}\\
&\sum_{n=0}^{\infty}\frac{
\left(-q;q\right)_{n}\left(-q^2;q^2\right)_{n-1}
}{\left(q;q\right)_{2n}}q^{\frac{n\left(n+1\right)}{2}}
\nonumber\\
&\quad
=\frac{\left(-q;q\right)_{\infty}}
{\left(q;q\right)_{\infty}}\left(\left(-q^{16},-q^{16},q^{32};q^{32}\right)_{\infty}
-q\left(-q^{8},-q^{24},q^{32};q^{32}\right)_{\infty}\right).\tag{S.104} \label{S.104}
\end{align}
\end{corollary}
\begin{proof}
When $\alpha=1,\beta={\rm i}$, Theorem \ref{T4} reduces to 
\begin{align}
&\frac{\left(q,ab/q;q\right)_{\infty}}{\left(a,b;q\right)_{\infty}}
\sum_{n=0}^{\infty}\frac{\left(q/a,q/b;q\right)_n\left(-1;q^2\right)_n}
{\left(q,-1;q\right)_{n}\left(q;q^2\right)_{n}}\left(ab/q\right)^n
\nonumber\\
&\quad=1+
\sum_{n=1}^{\infty}\left(-1\right)^n\left(1+q^{2n}\right)\frac{\left(q/a,q/b;q\right)_{2n}}
{\left(a,b;q\right)_{2n}}\left(ab\right)^{2n}q^{2n^2-3n}.\label{C13}
\end{align}
Then,  by taking $\left(a,b\right)\to\left(0,0\right)$ into \eqref{C13}, we can prove \eqref{S.54} as follows,
\begin{align*}
&\sum_{n=0}^{\infty}\frac{\left(-q^2;q^2\right)_{n-1}\left(1+q^n\right)}
{\left(q;q\right)_{2n}}q^{n^2}=
\sum_{n=0}^{\infty}\frac{\left(-1;q^2\right)_n}
{\left(q,-1;q\right)_{n}\left(q;q^2\right)_{n}}q^{n^2}=\frac{\left(q^5,q^7,q^{12};q^{12}\right)_{\infty}}{\left(q;q\right)_{\infty}}.
\end{align*}

Similarly, replacing $q$ by $q^2$ and inserting $\left(a,b\right)\to\left(0,-q\right)$ into \eqref{C13}, we arrive at \eqref{S.71}.

In addition, taking $\left(a,b\right)\to\left(0,0\right)$ into \eqref{C13}, we obtain 
\begin{align*}
&\sum_{n=0}^{\infty}\frac{
\left(-q;q\right)_{n}\left(-q^2;q^2\right)_{n-1}
}{\left(q;q\right)_{2n}}q^{\frac{n\left(n+1\right)}{2}}\\
&\quad
=
\frac{1}{2}+\frac{1}{2}\sum_{n=0}^{\infty}\frac{\left(-1;q^2\right)_n}
{\left(q;q\right)_n\left(q;q^2\right)_n}q^{\frac{n\left(n+1\right)}{2}}\\
&\quad=\frac{1}{2}+\frac{1}{2}\frac{\left(-q;q\right)_{\infty}\left(q^4,q^4,q^8;q^8\right)_{\infty}}{\left(q;q\right)_{\infty}}\\
&\quad=\frac{\left(-q;q\right)_{\infty}}
{\left(q;q\right)_{\infty}}\left(\left(-q^{16},-q^{16},q^{32};q^{32}\right)_{\infty}
-q\left(-q^{8},-q^{24},q^{32};q^{32}\right)_{\infty}\right)
\end{align*}
where we have used the following relations,
\begin{align*}
\frac{\left(q;q\right)_{\infty}}{\left(-q;q\right)_{\infty}}
&=\sum_{n=-\infty}^{\infty}\left(-1\right)^nq^{n^2}\\
&=\sum_{n=-\infty}^{\infty}q^{16n^2}
-2q\sum_{n=-\infty}^{\infty}q^{16n^2+8n}+q^4
\sum_{n=-\infty}^{\infty}q^{16n^2+16n}
\\
&=\left(-q^{16},-q^{16},q^{32};q^{32}\right)_{\infty}
-2q\left(-q^{8},-q^{24},q^{32};q^{32}\right)_{\infty}
+q^4\left(-1,-q^{32},q^{32};q^{32}\right)_{\infty}
\end{align*}
and
\begin{align*}
\left(q^4,q^4,q^8;q^8\right)_{\infty}=\left(-q^{16},-q^{16},q^{32};q^{32}\right)_{\infty}
-q^4\left(-1,-q^{32},q^{32};q^{32}\right)_{\infty}.
\end{align*}
Therefore, we prove \eqref{S.104} is true.
\end{proof}
\begin{corollary}
We have
\begin{align}
\sum_{n=0}^{\infty}\frac{\left(-q^2;q^2\right)_n}{\left(q;q\right)_{2n+1}\left(1+q^{n+1}\right)}q^{n^2+2n}
&=\frac{\left(q,q^{11},q^{12};q^{12}\right)_{\infty}}{\left(q;q\right)_{\infty}},
\tag{S.49}
\label{S.49}\\
\sum_{n=0}^{\infty}\frac{\left(-q^2;q^2\right)_n}{\left(q;q\right)_n\left(q;q^2\right)_{n+1}}q^{\frac{n\left(n+1\right)}{2}}
&=\frac{\left(-q;q\right)_{\infty}\left(q^2,q^6,q^8;q^8\right)_{\infty}}{\left(q;q\right)_{\infty}},\tag{\rm{\cite[Entry. 1.7.5]{Ramanujan}}}\label{N11}\\
\sum_{n=0}^{\infty}\frac{\left(-q;q^2\right)_{n+1}\left(-q^4;q^4\right)_n
q^{n^2+2n}}
{\left(-q^2;q^2\right)_{n+1}\left(q^2;q^2\right)_n\left(q^2;q^4\right)_{n+1}}
&=\frac{\left(-q;q^2\right)_{\infty}\left(q^2,q^{14},q^{16};q^{16}\right)_{\infty}}
{\left(q^2;q^2\right)_{\infty}}.\tag{S.68}\label{S.68}
\end{align}
\end{corollary}
\begin{proof}
When $\alpha=q^2,\beta=q{\rm i}$, Theorem \ref{T4} reduces to 
\begin{align}
&\frac{\left(q,abq;q\right)_{\infty}}{\left(aq^2,bq^2;q\right)_{\infty}}
\sum_{n=0}^{\infty}\frac{\left(q/a,q/b;q\right)_n\left(-q^2;q^2\right)_n}
{\left(q;q\right)_{n}\left(-q;q\right)_{n+1}\left(q;q^2\right)_{n+1}}\left(abq\right)^n
\nonumber\\
&\quad=
\sum_{n=0}^{\infty}\left(-1\right)^n\left(1-q^{2n+1}\right)\frac{\left(q/a,q/b;q\right)_{2n}}
{\left(aq^2,bq^2;q\right)_{2n}}\left(ab\right)^{2n}q^{2n^2+3n}.\label{C14}
\end{align}

Then, we take $\left(a,b\right)\to\left(0,0\right)$ and  $\left(a,b\right)\to\left(0,-1/q\right)$ into \eqref{C14}, respectively,  to obtain \eqref{S.49} and \ref{N11}.

In addition, replacing $q$ by $q^2$ and  setting  $\left(a,b\right)\to\left(0,-1/q\right)$ to \eqref{C14}, we get \eqref{S.68}.
\end{proof}

\begin{corollary}
We have
\begin{align}
\sum_{n=0}^{\infty}\frac{\left(-q;q^2\right)_{n+1}\left(q^4;q^4\right)_n
q^{n^2+2n}}
{\left(q^2;q^2\right)_{n+1}\left(q^2;q^2\right)_n\left(q^2;q^4\right)_{n+1}}
&=\frac{\left(-q;q^2\right)_{\infty}\left(-q^2,-q^{14},q^{16};q^{16}\right)_{\infty}}
{\left(q^2;q^2\right)_{\infty}}.\tag{S.69}\label{S.69}
\end{align}
\end{corollary}
\begin{proof}
Putting $\left(\alpha,\beta\right)\to\left(q^2,q\right)$ in Theorem \ref{T4}, we get
\begin{align}
&\frac{\left(q,abq;q\right)_{\infty}}{\left(aq^2,bq^2;q\right)_{\infty}}
\sum_{n=0}^{\infty}\frac{\left(q/a,q/b;q\right)_n\left(q^2;q^2\right)_n}
{\left(q;q\right)_n\left(q;q\right)_{n+1}\left(q;q^2\right)_{n+1}}\left(abq\right)^n\nonumber\\
&\quad=\sum_{n=0}^{\infty}\left(1+q^{2n+1}\right)\frac{\left(q/a,q/b;q\right)_{2n}}{\left(aq^2,bq^2;q\right)_{2n}}
\left(ab\right)^{2n}q^{2n^2+3n}.\label{C15}
\end{align}
Then, we  can obtain \eqref{S.69} by replacing $q$ by $q^2$ and letting $\left(a,b\right)\to\left(0,-1/q\right)$ in \eqref{C15}.
\end{proof}
\section{Identities from two ${_5\phi_4}$ summations of Andrews}\label{sec7}
In section, we utilize Andrew's  two ${_5\phi_4}$ summation formulas to establish some $q$-series transformations.
\begin{theorem}\label{theorem2}
For $|\alpha ab/q|<1$, we have
\begin{align*}
&\frac{\left(\alpha q,\alpha ab/q;q\right)_{\infty}}{\left(\alpha a,\alpha b;q\right)_{\infty}}
\sum_{n=0}^{\infty}
\frac{\left(q/a,q/b;q\right)_n\left(\alpha^{1/3} q^{1/3};q^{1/3}\right)_{3n}}{\left(q;q\right)_n\left(\alpha q;q\right)_{2n}}\left(\alpha ab/q\right)^n
\\
&\quad=
\sum_{n=0}^{\infty}\left(-1\right)^n
\frac{1-\alpha^{\frac{1}{3}}q^{\frac{2n}{3}}}{1-\alpha^{\frac{1}{3}}}
\frac{\left(q/a,q/b;q\right)_n\left(\alpha^{1/3};q^{1/3}\right)_n}
{\left(\alpha a,\alpha b;q\right)_n\left(q^{1/3};q^{1/3}\right)_n}\alpha^{\frac{4n}{3}}a^nb^nq^{\frac{3n^2-7n}{6}}.
\end{align*}
\end{theorem}
\begin{proof}
We begin by recalling the following strange $q$-series identity due to Andrews
\cite[Eq. (4.5)]{Andrews}:
\begin{align}
&{_5\phi_4}\left(\begin{array}{ccccc}
q^{-n},\,\alpha q^n,\,\alpha^{1/3}q^{1/3},\,\alpha^{1/3}q^{2/3},\,\alpha^{1/3}q\\
{\alpha}^{1/2}q,\,-\alpha^{1/2}q,\,\alpha^{1/2}q^{1/2},\,-\alpha^{1/2}q^{1/2}
\end{array};q,\,q
\right)\nonumber\\
&\quad=\frac{\left(1-\alpha\right)\left(1-\alpha^{1/3}q^{2n/3}\right)
\left(q;q\right)_n\left(\alpha^{1/3};q^{1/3}\right)_n}
{\left(1-\alpha^{1/3}\right)\left(1-\alpha q^{2n}\right)\left(\alpha;q\right)_n\left(q^{1/3};q^{1/3}\right)_n}\left(\alpha q\right)^{n/3}.
\label{Andrews2}
\end{align}
Then, we substitute $\beta=\alpha^{1/3}q^{1/3}, \beta=\alpha^{1/3}q^{2/3}, \lambda=\alpha^{1/3}q, c=-d=\alpha^{1/2}q$ and 
$e=-h=\alpha^{1/2}q^{1/2}$ into \eqref{Formula4} and employ \eqref{Andrews2} to obtain the desired result.
\end{proof}
\begin{corollary}
We have
\begin{align}
&\sum_{n=0}^{\infty}\frac{\left(q;q\right)_{3n+1}}
{\left(q;q^3\right)_n\left(q^3;q^3\right)_{2n+1}}q^{3n^2+3n}
=\frac{\left(q,q^8,q^9;q^9\right)_{\infty}}{\left(q^3;q^3\right)_{\infty}}
,\tag{S.40}\label{S.40}\\
&\sum_{n=0}^{\infty}\frac{\left(q;q\right)_{3n+1}\left(-q^3;q^3\right)_n
}
{\left(q^3;q^3\right)_n\left(q^3;q^3\right)_{2n+1}}
q^{\frac{{3n^2+3n}}{2}}
=\frac{\left(-q^3;q^3\right)_{\infty}
\left(q,q^5,q^6;q^6\right)_{\infty}}{\left(q^3;q^3\right)_{\infty}}.\tag{\cite[Eq. (7.3)]{Bailey}}\label{new5}
\end{align}
\end{corollary}
\begin{proof}
Taking replacing  $q$ by $q^6$ and  fixing $\alpha=q^6$ in Theorem \ref{theorem2}, we deduce that,
for $\max\left\{|ab|, |aq^6|, |bq^6|\right\}<1$, 
\begin{align}
&\frac{\left(q^{12},ab;q^6\right)_{\infty}}{\left(aq^6,bq^6;q^6\right)_{\infty}}
\sum_{n=0}^{\infty}
\frac{\left(q^6/a,q^6/b,q^4,q^8;q^6\right)_n}
{\left(q^6,-q^6,q^9,-q^9;q^6\right)_n}\left(ab\right)^n\nonumber\\
&\quad=\sum_{n=0}^{\infty}\left(-1\right)^n\frac{1-q^{4n+2}}{1-q^2}
\frac{\left(q^6/a,q^6/b;q^6\right)_n}{\left(aq^6,bq^6;q^6\right)_n}a^nb^nq^{3n^2+n}.\label{C18}
\end{align}
Taking $\left(a,b\right)\to\left(0,0\right)$ into \eqref{C18}, we obtain
\begin{align}
&\left(q^{12};q^6\right)_{\infty}
\sum_{n=0}^{\infty}
\frac{\left(q^4,q^8;q^6\right)_n}{\left(q^{12},q^{18};q^{12}\right)_n}
q^{6n^2+6n}
=\frac{\left(q^2,q^{16},q^{18};q^{18}\right)_{\infty}}{1-q^2}.\label{B2}
\end{align}
Then, replacing $q^2$ by $q$, we get
\begin{align*}
\sum_{n=0}^{\infty}\frac{\left(q;q^3\right)_{n+1}\left(q^2;q^3\right)_n}
{\left(q^6;q^3\right)_{2n}}q^{3n^2+3n}
=\frac{\left(q,q^8,q^9;q^9\right)_{\infty}}{
\left(q^6;q^3\right)_{\infty}},
\end{align*}
which means \eqref{S.40} is true.

Moreover, inserting $\left(a,b\right)\to\left(0,-1\right)$ into \eqref{C18}, we obtain
\begin{align}
&\frac{\left(q^{12};q^6\right)_{\infty}}{\left(-q^6;q^6\right)_{\infty}}
\sum_{n=0}^{\infty}
\frac{\left(q^4,q^8;q^6\right)_n}
{\left(q^6;q^6\right)_n\left(q^{18};q^{12}\right)_n}q^{3n^2+3n}
=\frac{\left(q^2,q^{10},q^{12};q^{12}\right)_{\infty}}{1-q^2}.\label{B3}
\end{align}
Then, letting $q^2\to q$ in \eqref{B3}, we have
\begin{align*}
\sum_{n=0}^{\infty}\frac{\left(q;q^3\right)_{n+1}\left(q^2;q^3\right)_n}
{\left(q^3;q^3\right)_n\left(q^3;q^6\right)_{n+1}}q^{\frac{3n^2+3n}{2}}
=
\frac{\left(-q^3;q^3\right)_n\left(q,q^{5},q^{6};q^{6}\right)_{\infty}}
{\left(q^3;q^3\right)_{\infty}},
\end{align*}
which means \eqref{new5} is true.
\end{proof}

\begin{corollary}
We have
\begin{align}
\sum_{n=0}^{\infty}\frac{\left(q;q\right)_{3n}}{\left(q^3;q^3\right)_n\left(q^3;q^3\right)_{2n}}q^{3n^2}
&=\frac{\left(q^4,q^5,q^9;q^9\right)_{\infty}}{\left(q^3;q^3\right)_{\infty}},\tag{S.42}\label{S.42}\\
\sum_{n=0}^{\infty}\frac{\left(q;q\right)_{3n}\left(-q^3;q^3\right)_n}{\left(q^3;q^3\right)_n\left(q^3;q^3\right)_{2n}}q^{\frac{3n^2-3n}{2}}
&=\frac{\left(-q^3;q^3\right)_{\infty}}{\left(q^3;q^3\right)_{\infty}}
\left(\left(q,q^5,q^6;q^6\right)_{\infty}+\left(q^2,q^4,q^6;q^6\right)_{\infty}\right),\label{N12}\\
\sum_{n=0}^{\infty}\frac{\left(q^2;q^2\right)_{3n}}{\left(q^3;q^3\right)_{2n}\left(q^{12};q^{12}\right)_{n}}q^{3n^2}
&=\frac{\left(-q^3;q^6\right)_{\infty}\left(q^5,q^7,q^{12};q^{12}\right)_{\infty}}{\left(q^6;q^6\right)_{\infty}}
.\tag{\rm{\cite[Eq. (7.5)]{Bailey}}}\label{D1}
\end{align}
\end{corollary}
\begin{proof}
Replacing $q$ by $q^3$ and fixing $\alpha=1$ in Theorem \ref{theorem2}, we have
\begin{align}
&\frac{\left(q^3,ab/q^3;q^3\right)_{\infty}}{\left(a,b;q^3\right)_{\infty}}
\sum_{n=0}^{\infty}\frac{\left(q^3/a,q^3/b;q^3\right)_n\left(q;q\right)_{3n}}
{\left(q^3;q^3\right)_n\left(q^3;q^3\right)_{2n}}\left(ab/q^3\right)^n\nonumber\\
&\quad=
1+\sum_{n=0}^{\infty}\left(-1\right)^n\left(1+q^n\right)\frac{\left(q^3/a,q^3/b;q^3\right)_n}{\left(a,b;q^3\right)_n}
a^nb^nq^{\frac{3n^2-7n}{2}}.\label{C19}
\end{align}

Taking $\left(a,b\right)\to\left(0,0\right)$ and $\left(a,b\right)\to\left(0,-1\right)$  into \eqref{C19}, respectively, we arrive at \eqref{S.42} and \eqref{N12}. Moreover,
Replacing $q$ by $q^2$ and taking $\left(a,b\right)\to\left(0,-q^3\right)$ into \eqref{C19}, we arrive at \ref{D1}.
\end{proof}
\begin{theorem}\label{thm4}
For $\max\left\{|a|, |b|, |ab/q|\right\}<1$ and  $\omega=\mathrm{e^{2\pi\mathrm{i}/3}}$, we have
\begin{align*}
&\frac{\left(q,ab/q;q\right)_{\infty}}{\left(a,b;q\right)_{\infty}}
\sum_{n=0}^{\infty}
\frac{\left(q/a,q/b,\omega,\omega^2;q\right)_n}
{\left(q,-1,q^{1/2},-q^{1/2};q\right)_n}\left(ab/q\right)^n\\
&\quad=1+\sum_{n=1}^{\infty}\left(-1\right)^n\left(1+q^{3n}\right)
\frac{\left(q/a,q/b;q\right)_{3n}}
{\left(a,b;q\right)_{3n}}
a^{3n}b^{3n}q^{\frac{9n^2-9n}{2}}\\
&\qquad+\frac{q^2}{2}\sum_{n=1}^{\infty}\left(-1\right)^n\left(1+q^{3n-1}\right)
\frac{\left(q/a,q/b;q\right)_{3n-1}}
{\left(a,b;q\right)_{3n-1}}
a^{3n-1}b^{3n-1}q^{\frac{9n^2-15n}{2}}\\
&\qquad+\frac{1}{2q}\sum_{n=0}^{\infty}\left(-1\right)^n\left(1+q^{3n+1}\right)
\frac{\left(q/a,q/b;q\right)_{3n+1}}
{\left(a,b;q\right)_{3n+1}}
a^{3n+1}b^{3n+1}q^{\frac{9n^2-3n}{2}}
.
\end{align*}
\end{theorem}
\begin{proof}
Recall a $_5\phi_4$ summation formula first proved by Andrews \cite{Andrews}:
\begin{align}
&{_5\phi_4}\left(\begin{array}{ccccc}
q^{-n},\,\alpha q^n,\,\alpha^{1/3},\,\alpha^{1/3}\omega,\,\alpha^{1/3}\omega^2\\
\sqrt{\alpha},\,-\sqrt{\alpha},\,\sqrt{q\alpha},\,-\sqrt{q\alpha}
\end{array};q,\,q
\right)\nonumber\\
&\quad=\left\{
\begin{array}{ll}
0,&\quad n\equiv1,2\pmod3,\\
\frac{\left(q;q\right)_n\left(\alpha;q^3\right)_{n/3}}
{\left(\alpha;q\right)_{n}\left(q^3;q^3\right)_{n/3}}\alpha^{n/3},&\quad n\equiv0\pmod3.
\end{array}
\right.\label{Andrews1}
\end{align}
Letting $\alpha=1$ in \eqref{Andrews1}, we deduce that
\begin{align}
{_4\phi_3}\left(
\begin{array}{cccc}
q^{-n},\,q^n,\,\omega,\,\omega^2\\
-1,\,q^{1/2},\,-q^{1/2}
\end{array};q,\,q\right)
=\left\{
\begin{array}{ll}
-1/2,&\quad n\equiv1,\,2\pmod3,\\
1,&\quad n\equiv0\pmod3.
\end{array}\right.\label{A3}
\end{align}
Then, substituting $\left(\alpha, \beta,\gamma,c,d,h\right)
\to\left(1,\omega,\omega^2,q^{1/2},-q^{1/2},-1\right)$ and $\lambda=e$ into \eqref{Formula4}, we obtain the desired result.
\end{proof}
\begin{corollary}
We have
\begin{align}
&\sum_{n=0}^{\infty}\frac{\left(-q;q\right)_n\left(q^3;q^3\right)_{n-1}
}
{\left(q;q\right)_{n}\left(q;q\right)_{2n-1}}q^{\frac{n(n-1)}{2}}\nonumber\\
&\quad=\frac{\left(-q;q\right)_{\infty}}{\left(q;q\right)_{\infty}}
\left(
\left(q^{6},q^{12},q^{18};q^{18}\right)_{\infty}+
\left(q^{9},q^{9},q^{18};q^{18}\right)_{\infty}\right)
,\tag{S.73}\label{S.73}\\
&\sum_{n=0}^{\infty}\frac{\left(-q;q\right)_n\left(q^3;q^3\right)_{n-1}
}
{\left(q;q\right)_{n-1}\left(q;q\right)_{2n}}q^{\frac{n(n+1)}{2}}\nonumber\\
&\quad=\frac{\left(-q;q\right)_{\infty}}{\left(q;q\right)_{\infty}}
\left(
\left(q^{9},q^{9},q^{18};q^{18}\right)_{\infty}-q
\left(q^{3},q^{15},q^{18};q^{18}\right)_{\infty}\right)
,\tag{S.75}\label{S.75}\\
&1+\sum_{n=1}^{\infty}\frac{\left(-1;q\right)_{n+1}\left(q^3;q^3\right)_{n-1}
}
{\left(q;q\right)_{n-1}\left(q;q\right)_{2n}}q^{\frac{n(n+1)}{2}}
=\frac{\left(-q;q\right)_{\infty}\left(q^{9},q^{9},q^{18};q^{18}\right)_{\infty}}{\left(q;q\right)_{\infty}}
,\tag{S.78}\label{S.78}\\
&\sum_{n=0}^{\infty}\frac{\left(q^3;q^3\right)_{n-1}}
{\left(q;q\right)_n\left(q;q\right)_{2n-1}}q^{n^2}
=\frac{\left(q^{12},q^{15},q^{27};q^{27}\right)_{\infty}}{\left(q;q\right)_{\infty}}
,\tag{S.93}\label{S.93}\\
&\sum_{n=0}^{\infty}\frac{\left(-q;q^2\right)_{n}
\left(q^6;q^6\right)_{n-1}}
{\left(q^2;q^2\right)_{2n-1}\left(q^2;q^2\right)_n}q^{n^2}
=\frac{\left(-q;q^2\right)_{\infty}\left(q^{15},q^{21},q^{36};q^{36}\right)_{\infty}}
{\left(q^2;q^2\right)_{\infty}}.\tag{S.114} \label{S.114}
\end{align}
\end{corollary}
\begin{proof}
We observe that
\begin{align}
&\sum_{n=0}^{\infty}\frac{\left(-q;q\right)_n\left(q^3;q^3\right)_{n-1}
}
{\left(q;q\right)_{n}\left(q;q\right)_{2n-1}}q^{\frac{n(n-1)}{2}}\nonumber\\
&\quad=1+\sum_{n=1}^{\infty}\frac{\left(-q;q\right)_n
\left(q\omega,q\omega^2;q\right)_{n-1}}
{\left(q,q^{1/2},-q^{1/2};q\right)_n\left(-q;q\right)_{n-1}}
q^{\frac{n\left(n-1\right)}{2}}\nonumber\\
&\quad=1+\frac{2}{3}\sum_{n=1}^{\infty}\frac{\left(-q,\omega,\omega^2;q\right)_{n}}
{\left(q,-1,q^{1/2},-q^{1/2};q\right)_n}
q^{\frac{n\left(n-1\right)}{2}}\nonumber\\
&\quad=\frac{1}{3}+\frac{2}{3}\sum_{n=0}^{\infty}\frac{\left(-q,\omega,\omega^2;q\right)_{n}}
{\left(q,-1,q^{1/2},-q^{1/2};q\right)_n}
q^{\frac{n\left(n-1\right)}{2}}.\label{A6}
\end{align}
Meanwhile, taking $\left(a,b\right)\to\left(0,-1\right)$ into Theorem \ref{thm4} results in
\begin{align}
&\frac{\left(q;q\right)_{\infty}}{\left(-1;q\right)_{\infty}}
\sum_{n=0}^{\infty}\frac{\left(-q,\omega,\omega^2;q\right)_n}
{\left(q,-1,q^{1/2},-q^{1/2};q\right)_n}q^{\frac{n\left(n-1\right)}{2}}\nonumber\\
&\quad=1+\frac{1}{2}\sum_{n=1}^{\infty}\left(-1\right)^n\left(1+q^{3n}\right)^2
q^{9n^2-3n}
+\frac{q^2}{4}\sum_{n=1}^{\infty}\left(-1\right)^n\left(1+q^{3n-1}\right)^2
q^{9n^2-9n}\nonumber\\
&\qquad+\frac{1}{4}\sum_{n=0}^{\infty}\left(-1\right)^n\left(1+q^{3n+1}\right)^2
q^{9n^2+3n}\nonumber\\
&\quad=\frac{1}{2}\sum_{n=-\infty}^{\infty}\left(-1\right)^nq^{9n^2}
+\frac{q}{2}\sum_{n=-\infty}^{\infty}\left(-1\right)^nq^{9n^2-6n}
+\frac{3}{4}\sum_{n=-\infty}^{\infty}\left(-1\right)^nq^{9n^2-3n}.\label{A7}
\end{align}
Then, by inserting \eqref{A7} into \eqref{A6}, we deduce that
\begin{align}
&\frac{\left(q;q\right)_{\infty}}{\left(-1;q\right)_{\infty}}
\sum_{n=0}^{\infty}\frac{\left(-q;q\right)_n\left(q^3;q^3\right)_{n-1}
}
{\left(q;q\right)_{n}\left(q;q\right)_{2n-1}}q^{\frac{n(n-1)}{2}}\nonumber\\
&\quad=\frac{1}{6}\frac{\left(q;q\right)_{\infty}}{\left(-1;q\right)_{\infty}}
+\frac{1}{3}\left(q^9,q^9,q^{18};q^{18}\right)_{\infty}
+\frac{q}{3}\left(q^3,q^{15},q^{18};q^{18}\right)_{\infty}
+\frac{1}{2}\left(q^6,q^{12},q^{18};q^{18}\right)_{\infty}
\nonumber\\
&\quad=\frac{1}{2}\left(q^9,q^9,q^{18};q^{18}\right)_{\infty}
+\frac{1}{2}\left(q^6,q^{12},q^{18};q^{18}\right)_{\infty},
\end{align}
where we have utilized the following fact:
\begin{align}
\frac{\left(q;q\right)_{\infty}}{\left(-q;q\right)_{\infty}}
&=\sum_{n=-\infty}^{\infty}\left(-1\right)^nq^{n^2}\nonumber\\
&=\sum_{n=-\infty}^{\infty}\left(-1\right)^nq^{9n^2}
-\sum_{n=-\infty}^{\infty}\left(-1\right)^nq^{\left(3n+1\right)^2}
-\sum_{n=-\infty}^{\infty}\left(-1\right)^nq^{\left(3n-1\right)^2}\nonumber\\
&=\left(q^9,q^9,q^{18};q^{18}\right)_{\infty}
-2q\left(q^3,q^{15},q^{18};q^{18}\right)_{\infty}.\label{A8}
\end{align}
Hence, \eqref{S.73} is true.

Next,  we take $\left(a,b\right)\to\left(0,-q\right)$ into Theorem \ref{thm4} to find that 
\begin{align}
&\frac{\left(q;q\right)_{\infty}}{\left(-q;q\right)_{\infty}}
\sum_{n=0}^{\infty}\frac{\left(\omega,\omega^2;q\right)_n}
{\left(q,q^{1/2},-q^{1/2};q\right)_n}
q^{\frac{n\left(n+1\right)}{2}}\nonumber\\
&\quad=1+2\sum_{n=1}^{\infty}\left(-1\right)^nq^{9n^2}
+\sum_{n=-\infty}^{\infty}\left(-1\right)^nq^{\left(3n+1\right)^2}\nonumber\\
&\quad=\left(q^9,q^9,q^{18};q^{18}\right)_{\infty}
+q\left(q^3,q^{15},q^{18};q^{18}\right)_{\infty}.\label{A4}
\end{align}
Also, we observe that
\begin{align}
&\frac{\left(q;q\right)_{\infty}}{\left(-q;q\right)_{\infty}}\sum_{n=0}^{\infty}\frac{\left(-q;q\right)_n\left(q^3;q^3\right)_{n-1}
}
{\left(q;q\right)_{n-1}\left(q;q\right)_{2n}}q^{\frac{n(n+1)}{2}}\nonumber\\
&\quad=\frac{\left(q;q\right)_{\infty}}{\left(-q;q\right)_{\infty}}+\frac{\left(q;q\right)_{\infty}}{\left(-q;q\right)_{\infty}}\sum_{n=1}^{\infty}\frac{\left(q\omega,q\omega^2;q\right)_{n-1}}
{\left(q,q^{1/2},-q^{1/2};q\right)_n}q^{\frac{n\left(n+1\right)}{2}}\nonumber\\
&\quad=\frac{\left(q;q\right)_{\infty}}{\left(-q;q\right)_{\infty}}+\frac{\left(q;q\right)_{\infty}}{3\left(-q;q\right)_{\infty}}\sum_{n=1}^{\infty}\frac{\left(\omega,\omega^2;q\right)_{n}}
{\left(q,q^{1/2},-q^{1/2};q\right)_n}q^{\frac{n\left(n+1\right)}{2}}\nonumber\\
&\quad=\frac{2\left(q;q\right)_{\infty}}{3\left(-q;q\right)_{\infty}}+\frac{\left(q;q\right)_{\infty}}{3\left(-q;q\right)_{\infty}}\sum_{n=0}^{\infty}
\frac{\left(\omega,\omega^2;q\right)_{n}}
{\left(q,q^{1/2},-q^{1/2};q\right)_n}q^{\frac{n\left(n+1\right)}{2}}\nonumber\\
&\quad=\frac{2}{3}\frac{\left(q;q\right)_{\infty}}{\left(-q;q\right)_{\infty}}
+\frac{1}{3}\left(\left(q^9,q^9,q^{18};q^{18}\right)_{\infty}
+q\left(q^3,q^{15},q^{18};q^{18}\right)_{\infty}
\right)\nonumber\\
&\quad=\left(q^9,q^9,q^{18};q^{18}\right)_{\infty}
-q\left(q^3,q^{15},q^{18};q^{18}\right)_{\infty},\nonumber
\end{align}
where in the last step we have used \eqref{A8}.
Therefore, we complete the proof of \eqref{S.75}.

Additionally, we can prove \eqref{S.78} as follows,
\begin{align*}
&1+\sum_{n=1}^{\infty}\frac{\left(q^3;q^3\right)_{n-1}\left(-1;q\right)_{n+1}}
{\left(q;q\right)_{2n}\left(-q;q\right)_{n-1}}q^{\frac{n(n+1)}{2}}\\
&\quad=
\frac{1}{3}+\frac{2}{3}\sum_{n=0}^{\infty}\frac{\left(\omega,\omega^2;q\right)_n}{\left(q,q^{1/2},-q^{1/2};q\right)_n}q^{\frac{n(n+1)}{2}}
\\
&\quad=\frac{\left(-q;q\right)_{\infty}\left(q^9,q^9,q^{18};q^{18}\right)_{\infty}}{\left(q;q\right)_{\infty}}.
\end{align*}

Taking $\left(a,b\right)\to\left(0,0\right)$ into Theorem \ref{thm4}, we immediately obtain
\begin{align}
&\left(q;q\right)_{\infty}\sum_{n=0}^{\infty}
\frac{\left(\omega,\omega^2;q\right)_n}
{\left(q,-1,q^{1/2},-q^{1/2};q\right)_n}q^{n^2}\nonumber\\
&\quad=1+\sum_{n=1}^{\infty}\left(-1\right)^n\left(1+q^{3n}\right)
q^{\frac{27}{2}n^2-\frac{3}{2}n}
+\frac{q}{2}\sum_{n=0}^{\infty}\left(-1\right)^n\left(1+q^{3n+1}\right)
q^{\frac{27}{2}n^2+\frac{15}{2}n}\nonumber\\
&\qquad+\frac{q^2}{2}\sum_{n=1}^{\infty}\left(-1\right)^n\left(1+q^{3n-1}\right)
q^{\frac{27}{2}n^2-\frac{21}{2}n}\nonumber\\
&\quad=\left(q^{12},q^{15},q^{27};q^{27}\right)_{\infty}
+\frac{q}{2}\left(q^{6},q^{21},q^{27};q^{27}\right)_{\infty}
+\frac{q^2}{2}\left(q^{3},q^{24},q^{27};q^{27}\right)_{\infty}.\label{A9}
\end{align}
Therefore, we can prove \eqref{S.93} as follows,
\begin{align}
&\left(q;q\right)_{\infty}\sum_{n=0}^{\infty}
\frac{\left(q^3;q^3\right)_{n-1}}
{\left(q;q\right)_n\left(q;q\right)_{2n-1}}q^{n^2}\nonumber\\
&\quad=\left(q;q\right)_{\infty}\left(1+\sum_{n=1}^{\infty}
\frac{\left(q,q\omega,q\omega^2;q\right)_{n-1}}
{\left(q;q\right)_n\left(q^{1/2},-q^{1/2};q\right)_{n}
\left(-q;q\right)_{n-1}}q^{n^2}\right)\nonumber\\
&\quad=\left(q;q\right)_{\infty}\left(1+\frac{2}{3}\sum_{n=1}^{\infty}
\frac{\left(\omega,\omega^2;q\right)_{n}}
{\left(q;q\right)_n\left(q^{1/2},-q^{1/2};q\right)_{n}
\left(-1;q\right)_{n}}q^{n^2}\right)\nonumber\\
&\quad=\left(q;q\right)_{\infty}\left(\frac{1}{3}+\frac{2}{3}\sum_{n=0}^{\infty}
\frac{\left(\omega,\omega^2;q\right)_{n}}
{\left(q;q\right)_n\left(q^{1/2},-q^{1/2};q\right)_{n}
\left(-1;q\right)_{n}}q^{n^2}\right)\nonumber\\
&\quad=\left(q^{12},q^{15},q^{27};q^{27}\right)_{\infty},
\end{align}
where we have used
\begin{align}
\left(q;q\right)_{\infty}
&=\sum_{n=-\infty}^{\infty}\left(-1\right)^nq^{\frac{3n^2-n}{2}}\nonumber\\
&=\sum_{n=-\infty}^{\infty}\left(-1\right)^nq^{\frac{27n^2-3n}{2}}
-q\sum_{n=-\infty}^{\infty}\left(-1\right)^nq^{\frac{27n^2-15n}{2}}
-q^2\sum_{n=-\infty}^{\infty}\left(-1\right)^nq^{\frac{27n^2-21n}{2}}\nonumber\\
&=\left(q^{12},q^{15},q^{27};q^{27}\right)_{\infty}
-q\left(q^{6},q^{21},q^{27};q^{27}\right)_{\infty}
-q^2\left(q^{3},q^{24},q^{27};q^{27}\right)_{\infty}.\label{A10}
\end{align}

Replacing $q$ by $q^2$ and taking $\left(a,b\right)\to\left(0,-q\right)$ into Theorem \ref{thm4}, we deduce that
\begin{align}
&\frac{\left(q^2;q^2\right)_{\infty}}{\left(-q;q^2\right)_{\infty}}
\sum_{n=0}^{\infty}\frac{\left(-q,\omega,\omega^2;q^2\right)_n}
{\left(q^2,-1,q,-q;q^2\right)_n}q^{n^2}
\nonumber\\
&\quad=1+\sum_{n=0}^{\infty}\left(-1\right)^n\left(1+q^{6n}\right)q^{18n^2-3n}
+\frac{q}{2}\sum_{n=0}^{\infty}\left(-1\right)^n\left(1+q^{6n+2}\right)q^{18n^2+9n}
\nonumber\\
&\qquad+\frac{q^3}{2}\sum_{n=0}^{\infty}\left(-1\right)^n
\left(1+q^{6n-2}\right)q^{18n^2-15n}\nonumber\\
&\quad=\left(q^{15},q^{21},q^{36};q^{36}\right)_{\infty}
+\frac{q}{2}\left(q^{9},q^{27},q^{36};q^{36}\right)_{\infty}
+\frac{q^3}{2}\left(q^{3},q^{33},q^{36};q^{36}\right)_{\infty}.\label{A11}
\end{align}
In terms of \eqref{A11}, we see that
\begin{align*}
&\frac{\left(q^2;q^2\right)_{\infty}}{\left(-q;q^2\right)_{\infty}}
\sum_{n=0}^{\infty}\frac{\left(-q;q^2\right)_n\left(q^6;q^6\right)_{n-1}}
{\left(q^2;q^2\right)_{2n-1}\left(q^2;q^2\right)_n}q^{n^2}\\
&\quad=
\frac{\left(q^2;q^2\right)_{\infty}}{\left(-q;q^2\right)_{\infty}}
\left(1+\sum_{n=1}^{\infty}\frac{\left(-q;q^2\right)_n
\left(q^2\omega,q^2\omega^2;q^6\right)_{n-1}}
{\left(q,-q,q^2;q^2\right)_{n}\left(-q^2;q^2\right)_{n-1}}q^{n^2}
\right)\\
&\quad=
\frac{\left(q^2;q^2\right)_{\infty}}{\left(-q;q^2\right)_{\infty}}
\left(1+\frac{2}{3}\sum_{n=1}^{\infty}\frac{\left(-q;q^2\right)_n
\left(\omega,\omega^2;q^6\right)_{n}}
{\left(-1,q,-q,q^2;q^2\right)_{n}}q^{n^2}
\right)\\
&\quad=
\frac{\left(q^2;q^2\right)_{\infty}}{\left(-q;q^2\right)_{\infty}}
\left(\frac{1}{3}+\frac{2}{3}\sum_{n=0}^{\infty}\frac{\left(-q;q^2\right)_n
\left(\omega,\omega^2;q^6\right)_{n}}
{\left(-1,q,-q,q^2;q^2\right)_{n}}q^{n^2}
\right)\\
&\quad=\frac{\left(q^2;q^2\right)_{\infty}}{3\left(-q;q^2\right)_{\infty}}
+\frac{2}{3}\left(q^{15},q^{21},q^{36};q^{36}\right)_{\infty}
+\frac{q}{3}\left(q^{9},q^{27},q^{36};q^{36}\right)_{\infty}
+\frac{q^3}{3}\left(q^{3},q^{33},q^{36};q^{36}\right)_{\infty}\\
&\quad=\left(q^{15},q^{21},q^{36};q^{36}\right)_{\infty},
\end{align*}
where we have used
\begin{align*}
\frac{\left(q^2;q^2\right)_{\infty}}{\left(-q;q^2\right)_{\infty}}
&=\sum_{n=-\infty}^{\infty}\left(-1\right)^nq^{2n^2-n}\\
&=\sum_{n=-\infty}^{\infty}\left(-1\right)^nq^{18n^2-3n}
-q\sum_{n=-\infty}^{\infty}\left(-1\right)^nq^{18n^2+9n}
-q^3\sum_{n=-\infty}^{\infty}\left(-1\right)^nq^{18n^2+15n}
\\
&=\left(q^{15},q^{21},q^{36};q^{36}\right)_{\infty}
-q\left(q^{9},q^{27},q^{36};q^{36}\right)_{\infty}
-q^3\left(q^{3},q^{33},q^{36};q^{36}\right)_{\infty}.
\end{align*}
Therefore, we complete the proof of \eqref{S.114}.
\end{proof}
\begin{theorem}\label{T1}
For $\max\left\{|aq^3|,|bq^3|,|abq^2|\right\}<1$, we have
\begin{align*}
&\frac{\left(q^3,abq^2;q\right)_{\infty}}{\left(aq^3,bq^3,q\right)_{\infty}}
\sum_{n=0}^{\infty}
\frac{\left(q/a,q/b;q\right)_n\left(q^3;q^3\right)_n}
{\left(q;q\right)_n\left(q^3;q\right)_{2n}}a^nb^nq^{2n}\\
&\quad=
\sum_{n=0}^{\infty}\left(-1\right)^n\left(1-q^{6n+3}\right)
\frac{\left(q/a,q/b;q\right)_{3n}}{\left(aq^3,bq^3;q\right)_{3n}}a^{3n}b^{3n}
q^{\frac{9n^2+15n}{2}}.
\end{align*}
\end{theorem}
\begin{proof}
Taking $\alpha=q^3$ into \eqref{Andrews1}, we get
\begin{align}
&{_5\phi_4}\left(\begin{array}{ccccc}
q^{-n},\,q^{n+3},\,q,\,q\omega,\,q\omega^2\\
q^{3/2},\,-q^{3/2},\,q^2,\,-q^2
\end{array};q,\,q
\right)\nonumber\\
&\quad=\left\{
\begin{array}{ll}
0,\,&n\equiv1,2\pmod3,\\
\frac{\left(1-q\right)\left(1-q^2\right)}
{\left(1-q^{3n+1}\right)\left(1-q^{3n+2}\right)}q^n,\, &n\equiv0\pmod3.
\end{array}
\right.\label{A12}
\end{align}
Setting $\left(\alpha,\beta,\gamma,\lambda,c,d,e,h\right)
=\left(q^3,q,q\omega,q\omega^2,q^{3/2},-q^{3/2},q^2,-q^2\right)$ in \eqref{Formula4} and using \eqref{A12}, we immediately obtain Theorem \ref{T1}.
\end{proof}
\begin{corollary}
We have
\begin{align}
\sum_{n=0}^{\infty}\frac{\left(-q;q\right)_{n+1}\left(q^3;q^3\right)_{n}}
{\left(q;q\right)_n\left(q;q\right)_{2n+2}}q^{\frac{n^2+3n}{2}}
&=\frac{\left(-q;q\right)_{\infty}\left(q^{3},q^{15},q^{18};q^{18}\right)_{\infty}}{\left(q;q\right)_{\infty}}
,\tag{S.76}\label{S.76}
\\
\sum_{n=0}^{\infty}\frac{\left(q^3;q^3\right)_{n}}
{\left(q;q\right)_n\left(q;q\right)_{2n+2}}q^{n^2+3n}
&=\frac{\left(q^{3},q^{24},q^{27};q^{27}\right)_{\infty}}{\left(q;q\right)_{\infty}}
,\tag{S.90}\label{S.90}
\\
\sum_{n=0}^{\infty}\frac{\left(-q;q^2\right)_{n+1}
\left(q^6;q^6\right)_n}
{\left(q^2;q^2\right)_{2n+2}\left(q^2;q^2\right)_n}q^{n^2+4n}
&=\frac{\left(-q;q^2\right)_{\infty}\left(q^3,q^{33},q^{36};q^{36}\right)_{\infty}}
{\left(q^2;q^2\right)_{\infty}}.\tag{S.116}\label{S.116}
\end{align}
\end{corollary}
\begin{proof}
Taking $\left(a,b\right)\to\left(0,-1/q\right)$ and $\left(a,b\right)\to\left(0,0\right)$, respectively,  into Theorem \ref{T1}, we obtain \eqref{S.76} and \eqref{S.90} after simplication.

Replacing $q$ by $q^2$ and taking $\left(a,b\right)\to\left(0,-1/q\right)$ into Theorem \ref{T1}, we get \eqref{S.116}.
\end{proof}
\begin{section}{Conclusion}
In this paper,  by reproving  $75$ identities of Rogers-Ramanujan type proposed by Slater and  uncovering several novel identities within this class, we show  that Liu's formula \eqref{Formula3} can imply many Rogers-Ramanujan type identities.  The readers who have interests can use  Liu's formula \eqref{Formula3} to discover more Rogers-Ramanujan type identities.

Below in Table $\ref{table1}$, we list the theorems and the corresponding  Rogers-Ramanujan type identities proposed by Slater.
\begin{table}[thbp]
\caption{Theorems and corresponding Rogers-Ramanujan type identities}\label{table1}
\centering
\begin{tabularx}{\textwidth}{lX}
  \toprule
  {Theorem \ref{thm9}} & $\eqref{S.2}$, $\eqref{S.3}$, $\eqref{S.5}$, $\eqref{S.7}$, $\eqref{S.8}$, $\eqref{S.9}$, $\eqref{S.12}$, $\eqref{S.13}$, $\eqref{S.14}$, $\eqref{S.18}$, $\eqref{S.24}$, $\eqref{S.25}$, $\eqref{S.28}$, $\eqref{S.29}$, $\eqref{S.30}$, $\eqref{S.34}$, $\eqref{S.36}$, $\eqref{S.38}$, $\eqref{S.39}$, $\eqref{S.47}$, $\eqref{S.50}$, $\eqref{S.51}$, $\eqref{S.52}$, $\eqref{S.64}$, $\eqref{S.84}$, $\eqref{S.86}$ \\
  Theorem \ref{T7} & $\eqref{S.45}$, $\eqref{S.60}$ \\
  Theorem \ref{T9} & $\eqref{S.96}$ \\
  Theorem \ref{thm2} & $\eqref{S.11}$, $\eqref{S.22}$, $\eqref{S.58}$, $\eqref{S.72}$, $\eqref{S.87}$ \\
  Theorem \ref{thm11} & $\eqref{S.16}$, $\eqref{S.17}$, $\eqref{S.20}$, $\eqref{S.31}$, $\eqref{S.32}$, $\eqref{S.94}$, $\eqref{S.99}$ \\
  Theorem \ref{T5} & $\eqref{S.48}$, $\eqref{S.65}$, $\eqref{S.67}$, $\eqref{S.70}$, $\eqref{S.127}$ \\
  Theorem \ref{T3} & $\eqref{S.21}$, $\eqref{S.33}$ \\
  Theorem \ref{T6} & $\eqref{S.53}$, $\eqref{S.55}$, $\eqref{S.57}$ \\
  Theorem \ref{T4} & $\eqref{S.35}$, $\eqref{S.43}$, $\eqref{S.49}$, $\eqref{S.54}$, $\eqref{S.59}$, $\eqref{S.61}$, $\eqref{S.68}$, $\eqref{S.69}$, $\eqref{S.71}$, $\eqref{S.98}$,  $\eqref{S.104}$ \\
  Theorem \ref{theorem2} & $\eqref{S.40}$, $\eqref{S.42}$\\
  Theorem \ref{thm4} & $\eqref{S.73}$, $\eqref{S.75}$, $\eqref{S.78}$, $\eqref{S.93}$, $\eqref{S.114}$\\
  Theorem \ref{T1} & $\eqref{S.76}$, $\eqref{S.90}$, $\eqref{S.116}$\\
  \bottomrule
\end{tabularx}
\end{table}
Astute readers may be 
aware that we have skipped  $55$ identities of Rogers-Ramanujan type proposed by Slater.  We omit the proof of $\left(\mathrm{S}.1\right)$ since it is a special case of  the Jacobi triple product identity \eqref{Jacobi}.  Challenges arise when we attempt to apply the preceding methodology to the remaining $54$ identities. The primary difficulties lie in evaluating certain terminated  $_5\phi_4$, $_4\phi_3$ and $_3\phi_2$ series. For instance, for Slater's identity (S.6), we take $\left(\alpha,\beta,a,b,c,d\right)\to\left(1,-1,0,0,q,0\right)$ into \eqref{Formula1} to deduce that
\begin{align}
\left(q;q\right)_{\infty}\sum_{n=0}^{\infty}
\frac{\left(-1;q\right)_n}{\left(q;q\right)^2_n}q^{n^2}
=1+\sum_{n=1}^{\infty}\left(-1\right)^n\left(1+q^n\right)q^{\frac{3n^2-n}{2}}
{_3\phi_2}\left(\begin{array}{ccc}
q^{-n},q^n,-1\\
q,0
\end{array};q,\,q\right).\label{B1}
\end{align}
However, we cannot find a useful method to evaluate the terminated
 $_3\phi_2$ series on the right hand-side of \eqref{B1}.  Settling the difficulties in calculating these kinds of terminated sums would be an interesting thing.
\end{section}
\section{Acknowledgment}
This work is supported by the National Natural Science Foundation of China (Grant 12371328).
The authors  would like to  express gratitude to Professor Zhiguo Liu for his helpful comments during the writing of this manuscript.

\end{document}